\documentclass[11pt]{amsart}

\usepackage[a4paper,margin=1in]{geometry}
\usepackage{amsmath,amssymb,amsthm,mathtools}
\usepackage{booktabs}
\usepackage{graphicx}
\usepackage{enumitem}
\usepackage{hyperref}
\usepackage{tikz}
\usetikzlibrary{arrows.meta,calc,positioning}

\hypersetup{colorlinks=true,linkcolor=blue,citecolor=blue,urlcolor=blue}

\newtheorem{theorem}{Theorem}[section]
\newtheorem{proposition}[theorem]{Proposition}
\newtheorem{lemma}[theorem]{Lemma}

\newtheorem{problem}[theorem]{Problem}
\newtheorem{question}[theorem]{Question}

\theoremstyle{definition}
\newtheorem{definition}[theorem]{Definition}
\newtheorem{remark}[theorem]{Remark}

\newcommand{\RR}{\mathbb R}
\newcommand{\ZZ}{\mathbb Z}
\newcommand{\SC}{\operatorname{SC}}
\newcommand{\BFACF}{\operatorname{BFACF}}
\newcommand{\Len}{\operatorname{Len}}

\newcommand{\Thi}{\operatorname{Thi}}
\newcommand{\Rop}{\operatorname{Rop}}
\newcommand{\diam}{\operatorname{diam}}

\newcommand{\Isomp}{\operatorname{Isom}^+}
\newcommand{\CRad}{\operatorname{CRad}}
\newcommand{\Rg}{R_g}
\newcommand{\sD}{\mathsf D}

\newcommand{\Lat}{\mathcal L}
\newcommand{\Pol}{\mathcal P}
\newcommand{\Comp}{\mathcal C}
\newcommand{\Move}{\mathcal M}
\newcommand{\Smooth}{\mathcal S}

\newcommand{\eps}{\varepsilon}
\newcommand{\seed}{\mathrm{seed}}

\title[Density and compression on merge trees]{Density and Compression on Lattice-Knot Merge Trees}
\author{Makoto Ozawa}
\address{Department of Natural Sciences, Komazawa University, Tokyo, Japan}
\email{w3c@komazawa-u.ac.jp}
\date{}

\begin{document}

\begin{abstract}
We introduce a finite-state framework for discrete $p$-density and
compression-radius profiles of lattice knots.  At a lattice length level $N$,
one considers lattice-polygon representatives of a fixed knot type, modulo
proper lattice isometries, and optionally restricts to the subgraph generated
from specified seeds by a local move system such as BFACF moves.  Raw density
and compression functionals are attached directly to the lattice polygons.  To
reduce lattice artefacts, the framework also permits controlled perturbation,
corner smoothing, thickness certification, and thickness normalization.

The unconditional results are finite-state statements: computability for a
fixed explored graph and finite smoothing scheme, monotonicity under increasing
length caps, and componentwise inequalities under graph mergers.  We further
organize the construction as a two-parameter filtration by lattice length and by
a density or compression threshold, giving a computable discrete counterpart of
continuous density- and compression-filtered knot spaces.  The analytic
lattice-to-continuous approximation problem is left open.

As preliminary numerical checks, we report raw-lattice computations for the
square unknot and a $24$-edge cubic-lattice trefoil, together with a restricted
length-filtration profile for the trefoil at $N=24,26,28,30,32$.  We then
compute raw geometric profiles on the complete minimal-layer merge trees of the
amphichiral knots $4_1$ and $6_3$ determined in the companion paper.  For every
birth component we evaluate the full distribution of $\rho_2$, $\rho_\infty$,
$\CRad_2$, and $\CRad_\infty$ on its minimal support.  We also evaluate the
entire bounded component at the first merge level: all $19618$ states of the
$4_1$ component at $N=32$, and both $12337$-state mirror components of $6_3$
at $N=42$.  Reflection gives identical geometric distributions on paired
branches.  The merge levels nevertheless enlarge the attainable geometric
ranges: for $4_1$ both density minima improve at $N=32$, while the two
compression minima remain unchanged; for $6_3$ all four displayed minima
improve at $N=42$.  The final $6_3$ merger at $N=44$ is represented by a
verified explicit certificate but its full bounded component is not claimed to
have been exhaustively enumerated.  Along the selected maximal-barrier
certificates, density and compression behave as competing rather than redundant
objectives.  All displayed values are explicitly raw-lattice quantities, not
exhaustive smoothed knot-type invariants.
\end{abstract}

\maketitle

%==================================================
\section{Introduction}
%==================================================

This paper develops a discrete, computable model for geometric density and
compression invariants of knot types.  The continuous $p$-density and
ropelength-windowed density theory is developed in~\cite{OzawaPDensity}, while
the general density--compression-radius factorization for scale-covariant size
functionals is developed in~\cite{OzawaCompression}.  The starting point here is
the continuous scale-free density
\[
 \rho_p(\gamma)=\frac{\Len(\gamma)}{\sD_p(\gamma)},
\]
where $\sD_p$ is an $L^p$-type mean chord spread, and the compression radius
\[
 \CRad_D(\gamma)=\frac{D(\gamma)}{\Thi(\gamma)}
\]
associated to a Euclidean size functional $D$.  These quantities are invariant
under similarity transformations of $\RR^3$.  When optimized over all
representatives of a knot type without geometric constraint, they may degenerate
because one can hide the knot in a small local region and let the rest of the
curve become long and nearly unknotted; this continuous degeneration is one of
the motivations for the ropelength-windowed theory in~\cite{OzawaPDensity}.
Ropelength-windowed variants prevent this degeneration by imposing thickness
normalization and a length bound.

The present paper asks how these quantities can be computed experimentally from
lattice knots.  The model is motivated by the companion lattice-filtered move
graph framework of~\cite{OzawaDiscrete}.  The lattice and BFACF background is
classical: standard BFACF moves are ergodic within each simple-cubic knot type
\cite{JvRW}, and minimal lattice lengths and minimal-layer populations have been
studied extensively \cite{DiaoMinimal,DiaoSmallest,Scharein2009,MinimalKnots}.
In particular, Scharein et al. enumerated minimal lattice knots and partitioned
them into classes under length-preserving BFACF type-0 moves
\cite{Scharein2009}.  Thus neither BFACF state spaces nor minimal-layer class
structure are claimed as new here.

The companion paper asks the quantitative minimax question left by this
background and, for the two principal amphichiral examples, determines the
complete minimal-layer hierarchy rather than only a prescribed seed pair
\cite{OzawaDiscrete}.  The $152$ minimal $4_1$ classes split at length $30$ into
four type-$0$ components of sizes $58,58,18,18$, all of which merge at $N=32$.
The $148$ minimal $6_3$ classes split at length $40$ into twelve components of
sizes $39,39,10,10,9,9,7,7,6,6,3,3$; these merge into two mirror-related
components at $N=42$ and into one component at $N=44$.  Thus the complete
length-barrier spectra on birth-level components are $\{0,2\}$ and
$\{0,2,4\}$, respectively.  The verified seed-to-mirror certificates and
coordinate data are archived in~\cite{OzawaDiscreteData}.  Here we do not
re-claim the merge-tree computation as the main contribution.  Instead, we
attach geometric density and compression functionals to the tree itself: first
to every birth-level component through its complete minimal support, and then to
the exhaustively enumerated bounded components at the first merge levels.  We
also retain selected certificates realizing the maximal barriers in order to
study pathwise geometric cost and the resulting geometric sublevel filtrations.

There are two levels of discretization.  The first is purely lattice-theoretic:
one evaluates $\rho_p$-type and $\CRad_D$-type quantities directly on lattice
polygons.  This is finite and transparent, but it retains strong artefacts of
the cubic lattice.  The second is a smoothed lattice model: each lattice polygon
is perturbed slightly and its corners are rounded, then the resulting
$C^{1,1}$ curve is rescaled to thickness one.  This is intended to give a closer
experimental analogue of the continuous ropelength-windowed definitions.  In the
present version, however, the smoothed stage is a certified protocol rather than
a source of large numerical tables.

The guiding principle is therefore
\[
\begin{aligned}
&\text{length filtration}+\text{local move graph}\\
&\qquad+\text{small perturbation}+\text{corner smoothing}
+\text{thickness normalization}.
\end{aligned}
\]
The resulting profiles are not claimed to be absolute invariants of the knot
type unless the lattice, move system, smoothing class, and search region are
specified.  Instead, they are reproducible finite-state experimental quantities.
This distinction is essential.  Global BFACF ergodicity guarantees eventual
connectivity inside a fixed knot type, but a bounded seed-generated search gives
information only about paths that remain under the chosen length cap; it does not
automatically enumerate the full bounded state graph.

The paper is written to be self-contained at the level of definitions.  The
lattice-filtered move graph, the density functionals, and the compression
functionals used below are all defined in the present text.  The continuous
papers~\cite{OzawaPDensity,OzawaCompression} identify the invariants being
discretized.  In particular, the density- and compression-sublevel filtrations
discussed in~\cite{OzawaCompression} motivate the two-parameter finite graphs
introduced below.  No continuous approximation theorem is used as an external
hypothesis in the finite-state arguments.  The elementary finite-state
propositions should be read as reproducibility certificates rather than as deep
existence theorems: the nontrivial analytic work in future experiments lies in
embeddedness checking, thickness certification, lattice approximation, and the
control of search bias.

\subsection*{Main contributions}

\begin{enumerate}[label=(\arabic*)]
\item We define raw and smoothed lattice $p$-density and compression-radius
profiles on filtered lattice-polygon sets and seed-generated move graphs.
\item We record finite-state computability for fixed length level and finite
smoothing scheme, together with monotonicity under increasing length caps.
\item We give an explicit deterministic corner-rounding scheme and formulate a
conservative interval-arithmetic pipeline for certified lower bounds on the
thickness of the resulting line-and-arc curves.
\item We place density- and compression-sublevel graphs inside the
lattice-length filtration.  This produces a computable two-parameter filtration
whose component births and mergers refine the global minimum profiles; the use of
nested component filtrations is standard, while the geometric vertex labels and
knot-theoretic optimization questions are specific to the present setting.
\item We include pilot raw-lattice computations for the square unknot and a
$24$-edge trefoil seed, together with a restricted $N$-increasing trefoil profile
at levels $24,26,28,30,32$.
\item We compute exact raw geometric profiles on every birth component of the
complete minimal-layer merge trees of $4_1$ and $6_3$.  We further compute the
full bounded profile of the unique $19618$-state $4_1$ component at $N=32$ and
of the two $12337$-state mirror components of $6_3$ at $N=42$.  Complete
state-level tables are supplied as reproducibility data.
\item We show that reflection-related branches have identical density and
compression distributions even when they remain disconnected in the bounded
BFACF graph.  We also retain verified maximal-barrier paths for $4_1$ and $6_3$
and show at the path level that density and compression can improve in opposite
directions.
\item We formulate lattice-to-continuous convergence and effective error bounds
only as open problems with scaled length windows, not as claimed theorems.
\item We relate componentwise density and compression minima to merge structure
and to threshold-filtered admissible components.
\end{enumerate}

\subsection*{Notation guide}
For reference we collect the main decorated symbols used in the paper.
\begin{center}
\small
\begin{tabular}{ll}
\toprule
symbol & meaning \\
\midrule
$\Lat$ & a periodic lattice; in computations $\SC=\ZZ^3$ \\
$\Pol_N^{\Lat}(K)$ & lattice polygons of type $K$ with $\ell(P)\le N$, modulo oriented lattice isometry \\
$G_N^{\Lat,\Move}(K;S)$ & seed-generated length-$N$ move graph from seeds $S$ \\
$G_{N,\Lambda}^{F}(K;S)$ & sublevel graph with vertex functional $F\leq\Lambda$ \\
$\sD_p$ & mean chord spread; $\sD_\infty=\diam$ \\
$\rho_p$ & scale-free density $\Len/\sD_p$ or $\ell/\sD_p$ in the raw lattice model \\
$\CRad_D$ & compression radius $D/\Thi$; in raw SC tables $\Thi_{\rm poly}=1/2$ is a labelled convention \\
$\Smooth(P)$ & finite smoothing/perturbation family attached to $P$ \\
$\rho_{p,N}^{\Lat}$, $\CRad_{D,N}^{\Lat}$ & raw lattice level-$N$ profiles \\
$\rho_{p,N}^{\Lat,\Move,\Smooth,\seed}$ & seed-generated smoothed level-$N$ profile \\
$\rho_{p,\lambda}^{\operatorname{rop}}$, $\CRad_{D,\lambda}^{\operatorname{rop}}$ & continuous ropelength-windowed profiles \\
\bottomrule
\end{tabular}
\end{center}

%==================================================
\section{Continuous density and compression functionals}\label{sec:continuous}
%==================================================

Let $\gamma$ be a tame embedded closed curve in $\RR^3$, parameterized by arclength.
Write $L=\Len(\gamma)$.  The definitions in this section are the continuous
quantities that the present lattice construction discretizes; see
\cite{OzawaPDensity} for the $p$-density family and ropelength windows, and
\cite{OzawaCompression} for compression radii associated with general
scale-covariant size functionals.

\begin{definition}[$p$-spread]
For $-1<p<\infty$, $p\ne0$, define
\[
 \sD_p(\gamma)=
 \left(\frac{1}{L^2}\int_\gamma\int_\gamma |x-y|^p\,ds_x\,ds_y\right)^{1/p}.
\]
For $p=0$, define $\sD_0(\gamma)$ by the logarithmic mean
\[
 \log \sD_0(\gamma)=
 \frac{1}{L^2}\int_\gamma\int_\gamma \log |x-y|\,ds_x\,ds_y,
\]
whenever the integral is finite.  For $p=\infty$, set
\[
 \sD_\infty(\gamma)=\diam(\gamma).
\]
\end{definition}

\begin{remark}[Why $p>-1$]
The lower bound $p>-1$ is the local integrability threshold along the diagonal.
Indeed, near a regular point of an embedded arclength-parametrized curve one has
$|\gamma(s)-\gamma(t)|\sim |s-t|$, so the diagonal contribution is locally of the
form
\[
 \int\!\int |s-t|^p\,ds\,dt,
\]
which is finite precisely for $p>-1$.  For $p=0$, the logarithmic singularity is
also locally integrable.  In particular, for a compact embedded $C^1$ curve, the
logarithmic mean is finite; polygonal representatives are handled edgewise, with
the same local integrability at points on the same edge.
\end{remark}

\begin{definition}[$p$-density]
The $p$-density of a representative $\gamma$ is
\[
 \rho_p(\gamma)=\frac{\Len(\gamma)}{\sD_p(\gamma)}.
\]
\end{definition}

Let $D$ be a Euclidean-invariant size functional on embedded curves, homogeneous
of degree one:
\[
 D(a\gamma)=aD(\gamma),\qquad a>0.
\]
Examples include $\diam$, the minimal enclosing radius, the radius of gyration,
and $\sD_p$.

\begin{definition}[Compression radius]
The $D$-compression radius of $\gamma$ is
\[
 \CRad_D(\gamma)=\frac{D(\gamma)}{\Thi(\gamma)}.
\]
If $D=\sD_p$, we write $\CRad_p(\gamma)=\sD_p(\gamma)/\Thi(\gamma)$.
\end{definition}

\begin{lemma}[Scale invariance]
For every $a>0$,
\[
 \rho_p(a\gamma)=\rho_p(\gamma),\qquad
 \CRad_D(a\gamma)=\CRad_D(\gamma).
\]
\end{lemma}

\begin{proof}
Both $\Len$ and $\sD_p$ scale by $a$.  Similarly, both $D$ and $\Thi$ scale by
$a$.  The ratios are therefore unchanged.
\end{proof}

\begin{definition}[Ropelength-windowed continuous profiles]
Let $K$ be a knot type and $\lambda\ge1$.  Define
\[
 \rho_{p,\lambda}^{\operatorname{rop}}(K)=
 \inf\left\{
 \frac{\Len(\gamma)}{\sD_p(\gamma)}\ \middle|\
 \gamma\in K,\ \Thi(\gamma)\ge1,
 \Len(\gamma)\le \lambda\Rop(K)
 \right\}.
\]
Likewise,
\[
 \CRad_{D,\lambda}^{\operatorname{rop}}(K)=
 \inf\left\{
 \frac{D(\gamma)}{\Thi(\gamma)}\ \middle|\
 \gamma\in K,\ \Thi(\gamma)>0,
 \frac{\Len(\gamma)}{\Thi(\gamma)}\le \lambda\Rop(K)
 \right\}.
\]
\end{definition}

\begin{remark}[Equivalent normalizations]
The two displayed windows use slightly different normalizations only for
notational convenience.  Since $\rho_p$ and $\CRad_D$ are scale invariant, a curve
with $\Thi(\gamma)>0$ and
\[
 \Len(\gamma)/\Thi(\gamma)\le \lambda\Rop(K)
\]
can be rescaled to thickness one without changing either functional.  Conversely,
under the condition $\Thi(\gamma)\ge 1$ and
$\Len(\gamma)\le\lambda\Rop(K)$, the ropelength bound
$\Len(\gamma)/\Thi(\gamma)\le\lambda\Rop(K)$ is automatic.  Thus both formulations
represent the same ropelength-window principle; the distinction is only a matter
of whether thickness one is imposed before or after evaluating a scale-free
quantity.
\end{remark}

The length window is essential.  Without it, density-type quantities may collapse
under local knotting constructions and cease to detect the knot type; the
continuous version of this degeneration and the ropelength-windowed refinement
are studied in~\cite{OzawaPDensity}.

%==================================================
\section{Lattice polygons and filtered move graphs}
%==================================================

Let $\Lat$ be a periodic lattice in $\RR^3$.  In the computational sections we
specialize to the simple cubic lattice $\SC=\ZZ^3$ with unit edges.

\begin{definition}[Lattice polygon]
A lattice polygon in $\Lat$ is an embedded closed polygonal curve whose edges are
lattice edges.  Its lattice length is the number of lattice edges and is denoted
$\ell(P)$.
\end{definition}

For a knot type $K$, write
\[
 \Pol_N^{\Lat}(K)=
 \{P\subset\Lat: P\text{ is a lattice polygon of type }K,
 \ell(P)\le N\}/\Isomp(\Lat).
\]
Here $\Isomp(\Lat)$ is the orientation-preserving lattice isometry group.  If
mirrors are to be identified, one may instead divide by the full lattice isometry
group.  In this paper mirrors are not identified unless explicitly stated.

\begin{definition}[Lattice length minimum]
The minimal lattice length of $K$ in $\Lat$ is
\[
 n_{\Lat}(K)=\min\{\ell(P):P\subset\Lat\text{ represents }K\}.
\]
\end{definition}

Let $\Move$ be a specified local move system on lattice polygons, for example a
BFACF-type move system in the simple cubic lattice; see
\cite{BergForster,ACCF,JvRW} for the origins and knot-theoretic use of the BFACF
algorithm.

\begin{definition}[Filtered move graph]
The lattice-filtered move graph at level $N$ is the finite graph
\[
 G_N^{\Lat,\Move}(K)
\]
whose vertices are $\Pol_N^{\Lat}(K)$ and whose edges connect two vertices when
they differ by one allowed move in $\Move$.
\end{definition}

\begin{definition}[Seed-generated graph]
Given a finite seed set $S\subset\Pol_{N_0}^{\Lat}(K)$, the seed-generated graph
\[
 G_N^{\Lat,\Move}(K;S)
\]
is the induced subgraph of $G_N^{\Lat,\Move}(K)$ generated by all vertices
reachable from $S$ through $\Move$-paths that remain at length at most $N$.
\end{definition}

\begin{remark}[Global BFACF ergodicity and bounded searches]
For the standard simple-cubic BFACF moves, the irreducibility classes on
unrooted polygons are exactly the knot types \cite{JvRW}.  Hence representatives
of the same knot type are eventually connected when no length cap is imposed.
This global theorem does not make a bounded search exhaustive: a BFACF path to a
polygon of length at most $N$ may have to pass through a longer intermediate
polygon.  Thus a seed-generated computation under the cap $N$ determines the
bounded components reachable from the chosen seeds, not automatically every
vertex of $G_N^{\Lat,\Move}(K)$.  In the simple-cubic BFACF setting, elementary
moves change the length by $0$ or $\pm2$; therefore a seed-generated search only
visits length levels with the same parity as the chosen seed length.
\end{remark}

\begin{definition}[Admissible length levels]
Given a move system $\Move$ and seed length $N_0$, let
\[
 \mathcal N_{\Move}(N_0)=\{N\ge N_0: N\text{ is reachable as a length level from }
 N_0\text{ under }\Move\}.
\]
For the simple-cubic BFACF searches used in the pilot computations,
$\mathcal N_{\BFACF}(N_0)=\{N_0,N_0+2,N_0+4,\ldots\}$.
\end{definition}

%==================================================
\section{Discrete lattice density and compression profiles}\label{sec:lattice-profiles}
%==================================================

The simplest discrete model evaluates the same pairwise-distance functionals on
the polygonal curve $P$ itself.

\begin{definition}[Lattice $p$-density at level $N$]
For a lattice polygon $P$, define
\[
 \rho_p^{\Lat}(P)=\frac{\ell(P)}{\sD_p(P)},
\]
where $\sD_p(P)$ is computed using arclength measure on the polygonal curve.
The level-$N$ lattice $p$-density of $K$ is
\[
 \rho_{p,N}^{\Lat}(K)=
 \min\{\rho_p^{\Lat}(P):P\in\Pol_N^{\Lat}(K)\}.
\]
For a seed set $S$, define
\[
 \rho_{p,N}^{\Lat,\Move,\seed}(K;S)=
 \min\{\rho_p^{\Lat}(P):P\in G_N^{\Lat,\Move}(K;S)\}.
\]
\end{definition}

\begin{definition}[Lattice compression radius at level $N$]
For a polygonal thickness $\Thi_{\operatorname{poly}}(P)$, set
\[
 \CRad_D^{\Lat}(P)=\frac{D(P)}{\Thi_{\operatorname{poly}}(P)}.
\]
Then define
\[
 \CRad_{D,N}^{\Lat}(K)=
 \min\{\CRad_D^{\Lat}(P):P\in\Pol_N^{\Lat}(K)\}
\]
and, for seed-generated computations,
\[
 \CRad_{D,N}^{\Lat,\Move,\seed}(K;S)=
 \min\{\CRad_D^{\Lat}(P):P\in G_N^{\Lat,\Move}(K;S)\}.
\]
\end{definition}

\begin{proposition}[Finite-state computability]\label{prop:finite-state-computability}
For fixed $\Lat$, $K$, $N$, and finite move-generated vertex set, the quantities
\[
 \rho_{p,N}^{\Lat,\Move,\seed}(K;S),\qquad
 \CRad_{D,N}^{\Lat,\Move,\seed}(K;S)
\]
are finite computable real numbers up to the numerical precision used to evaluate
the pairwise integrals and thickness.
If $\Pol_N^{\Lat}(K)$ is explicitly enumerated, the corresponding global level-$N$
quantities are also computable.
\end{proposition}

\begin{proof}
A seed-generated graph under a finite length cap has finitely many vertices.  The
functionals are evaluated on each vertex and a finite minimum is taken.  The same
argument applies to an exhaustive enumerated list of $\Pol_N^{\Lat}(K)$.
\end{proof}

\begin{proposition}[Monotonicity in the length level]\label{prop:monotonicity}
For $N\le N'$ one has
\[
 \rho_{p,N'}^{\Lat}(K)\le \rho_{p,N}^{\Lat}(K),\qquad
 \CRad_{D,N'}^{\Lat}(K)\le \CRad_{D,N}^{\Lat}(K),
\]
whenever both sides are defined.  The same holds for seed-generated profiles
provided the seed-generated vertex sets are nested with $N$.
\end{proposition}

\begin{proof}
The level-$N$ feasible set is contained in the level-$N'$ feasible set, so taking
minimum over the larger set cannot increase the value.
\end{proof}

\begin{remark}[Status of the finite-state propositions]
Proposition~\ref{prop:finite-state-computability} and
Proposition~\ref{prop:monotonicity} are deliberately elementary.  Their role is
not to supply a difficult existence theorem, but to distinguish three different
kinds of claims: a finite value certified on a specified explored graph, an
exhaustive value certified by a complete enumeration, and an unproved statement
about the full space of lattice representatives.  This distinction is essential
for the numerical tables below.
\end{remark}

\begin{remark}[Simple cubic thickness]
For a unit simple-cubic lattice polygon, the polygonal-thickness scale is affected
both by local corner geometry and by the global distance between non-adjacent
edges.  Under the standard polygonal-thickness convention, an orthogonal corner
contributes the local scale $1/2$, so for a polygon with a right-angle corner the
polygonal thickness is at most $1/2$; equality holds only when no non-adjacent
features force a smaller half-distance.  Thus non-local doubly critical distances
may also be relevant.  In the raw tables below we use the common simple-cubic
normalization
\[
 \Thi_{\rm poly}(P)=1/2,
 \qquad \CRad_D^{\SC}(P)=2D(P),
\]
only as a labelled lattice convention for the displayed representatives.  The
smoothed model avoids building the theory on this lattice-specific convention and
requires an explicit thickness certificate for each rounded representative.
\end{remark}

%==================================================
\section{Smoothed and perturbed lattice representatives}\label{sec:smoothing}
%==================================================

Direct lattice values are useful but still strongly lattice-dependent.  To obtain
quantities closer to the continuous ropelength-windowed invariants, we pass from
lattice polygons to smoothed, slightly perturbed curves.

\begin{definition}[Smoothing-perturbation scheme]
A smoothing-perturbation scheme $\Smooth=(\eps,r,\mathcal A)$ assigns to each
lattice polygon $P$ a nonempty finite set
\[
 \Smooth(P)=\{\gamma_1,\ldots,\gamma_m\}
\]
of embedded $C^{1,1}$ closed curves satisfying:
\begin{enumerate}[label=(\roman*)]
\item each $\gamma_i$ is ambient isotopic to $P$;
\item $\gamma_i$ lies in an $\eps$-tubular neighborhood of $P$;
\item corners of $P$ are replaced by arcs or splines of controlled radius at least
$r$, except where the algorithm certifies a larger local radius;
\item the construction is reproducible from the finite algorithmic data
$\mathcal A$.
\end{enumerate}
\end{definition}

\subsection{A concrete corner-rounding scheme}

We now record one deterministic smoothing scheme which can be implemented without
any auxiliary choices beyond a rounding radius.  Let
$P=(v_0,\ldots,v_{n-1})$ be an oriented unit-edge simple-cubic polygon.  At a
turning vertex $v_i$, put
\[
 a_i=v_i-v_{i-1},\qquad b_i=v_{i+1}-v_i,
\]
where indices are taken modulo $n$.  Thus $a_i$ is the incoming unit direction
and $b_i$ is the outgoing unit direction.  If $a_i\ne b_i$, choose
$0<r<1/2$ and replace the two trimmed subsegments ending at
\[
 q_i^- = v_i-r a_i,\qquad q_i^+=v_i+r b_i
\]
by the quarter-circular arc of radius $r$ in the coordinate plane spanned by
$a_i$ and $b_i$, centered at
\[
 c_i=v_i-r a_i+r b_i.
\]
The case \(a_i=-b_i\), corresponding to an immediate backtracking
edge, cannot occur in an embedded unit-edge lattice polygon after
the usual cancellation of repeated edges.  Thus the only local cases
for an embedded cubic-lattice polygon are a straight vertex
\(a_i=b_i\), where no rounding is performed, and a right-angle vertex
\(a_i\perp b_i\), where the above circular replacement is used.
The arc is chosen with tangent $a_i$ at $q_i^-$ and tangent $b_i$ at $q_i^+$.
Straight vertices are left unchanged.  We call the resulting curve the
\emph{$r$-rounded representative} and denote it by $\gamma_r(P)$.

A finite deterministic scheme is obtained by choosing a finite list
$r_1,\ldots,r_s$ of radii and setting
\[
 \Smooth_{\rm round}(P)=\{\gamma_{r_j}(P):1\le j\le s\},
\]
after discarding any rounded curve that fails an embeddedness certificate.  A
finite perturbed scheme is obtained by first replacing selected vertices by
rational displacements of size at most $\eps$ and then applying the same
rounding rule.  In that case each accepted curve must be accompanied by the
finite perturbation record and the rounding radius used.

\begin{proposition}[Elementary thickness certificate]
Let $P$ be a unit simple-cubic polygon and let $\gamma$ be obtained from $P$ by a
rounded perturbation lying in the $\delta$-neighborhood of $P$.  Suppose that all
non-adjacent local pieces of the smoothing are certified to have mutual distance
at least $1-2\delta$, and that all smoothing arcs have radius at least $r$.
Then
\[
 \Thi(\gamma)\ge \min\left\{r,\frac{1-2\delta}{2}\right\}.
\]
\end{proposition}

\begin{proof}
For a $C^{1,1}$ curve, thickness is the minimum of the minimum radius of
curvature and one half of the doubly critical self-distance.  The curvature
radius of each circular smoothing arc is at least $r$, and the straight pieces
have infinite curvature radius.  The assumed non-local separation gives a lower
bound $1-2\delta$ for the relevant self-distances.  Taking the minimum gives the
claim.
\end{proof}

\begin{remark}
The preceding proposition is intentionally conservative.  Its purpose is not to
produce sharp thickness values, but to ensure that every reported smoothed value
can be normalized by a certified positive lower bound for thickness.  The
non-adjacent-distance hypothesis is a genuine global condition, not an automatic
consequence of small rounding: densely packed lattice polygons may contain nearby
non-adjacent strands, and then the certificate can fail unless the distance check
is performed explicitly.  Sharper experiments should replace this bound by
interval subdivision or exact segment-arc distance calculations.
\end{remark}

\subsection{A finite thickness-certification pipeline}

For the rounded representatives used here, each curve is a finite union of line
segments and circular arcs.  Hence a practical thickness certificate can be
formulated without referring to a black-box smooth curve routine.  The following
procedure is the one intended for the smoothed experiments.

\begin{enumerate}[label=(C\arabic*)]
\item Decompose $\gamma$ into finitely many elementary pieces: line segments and
circular arcs.  Record their parameter intervals with rational or interval
endpoints.
\item For every adjacent pair, certify the prescribed tangent matching and record
the local curvature radius.  For the circular rounding scheme this lower bound is
at least the chosen radius $r$; for straight pieces it is infinite.
\item For every non-adjacent pair of elementary pieces, compute a certified lower
bound for the distance between the two parameterized pieces.  This can be done by
interval subdivision: on a rectangle of parameters $I\times J$, use interval
arithmetic to bound the range of
\[
 f(s,t)=|\gamma_i(s)-\gamma_j(t)|^2.
\]
If the interval lower bound is inconclusive, subdivide $I\times J$ until the
required tolerance is reached.
\item Take the minimum of the local radius bound and one half of the certified
non-local distance bound.  This gives a certified lower bound
\[
 \Thi(\gamma)\ge \tau_{\rm cert}>0.
\]
The value reported in numerical tables should be computed using the conservative
normalization by $\tau_{\rm cert}$ unless a sharper certified thickness value is
available.
\end{enumerate}

This gives a finite algorithm whenever the curve has a positive separation
margin and the subdivision tolerance is fixed in advance.  In implementations,
one may use standard interval arithmetic packages or a bounding-volume hierarchy
as an acceleration layer; these affect efficiency but not the mathematical
structure of the certificate.

\begin{proposition}[Certified computability for line-and-arc smoothing]
Fix a finite search graph, a finite list of rounding radii, and a fixed interval
subdivision tolerance.  Suppose each accepted rounded curve has a positive
certified separation margin.  Then the corresponding thickness-normalized
smoothed density and compression values are reproducibly computable with a
certified lower bound for thickness.
\end{proposition}

\begin{proof}
There are finitely many curves and finitely many pairs of elementary pieces for
each curve.  The local curvature-radius bounds are explicit.  The interval
subdivision step gives finite lower bounds for all non-adjacent piece distances
under the stated positive-margin hypothesis and fixed tolerance.  Combining these
bounds gives a positive certificate $\tau_{\rm cert}$, after which the density
and compression quantities are evaluated by finite quadrature or by closed-form
line-and-arc integrals to the prescribed accuracy.
\end{proof}

\begin{definition}[Thickness normalization]
For $\gamma\in\Smooth(P)$ with $\Thi(\gamma)>0$, define
\[
 \widehat\gamma=\frac{1}{\Thi(\gamma)}\gamma.
\]
Then $\Thi(\widehat\gamma)=1$.
\end{definition}

Because $\rho_p$ and $\CRad_D$ are scale invariant, the normalization is mainly a
way to put the curve into the same geometric scale as ropelength theory.  In
particular,
\[
 \rho_p(\widehat\gamma)=\rho_p(\gamma),
 \qquad
 \CRad_D(\widehat\gamma)=\CRad_D(\gamma).
\]

\begin{definition}[Smoothed lattice $p$-density]
For a finite smoothing-perturbation scheme $\Smooth$, define
\[
 \rho_{p,N}^{\Lat,\Smooth}(K)=
 \min\left\{
 \rho_p(\widehat\gamma)\ \middle|\
 P\in\Pol_N^{\Lat}(K),\ \gamma\in\Smooth(P)
 \right\}.
\]
For a seed set $S$ and move system $\Move$, define
\[
 \rho_{p,N}^{\Lat,\Move,\Smooth,\seed}(K;S)=
 \min\left\{
 \rho_p(\widehat\gamma)\ \middle|\
 P\in G_N^{\Lat,\Move}(K;S),\ \gamma\in\Smooth(P)
 \right\}.
\]
\end{definition}

\begin{definition}[Smoothed lattice compression radius]
Similarly,
\[
 \CRad_{D,N}^{\Lat,\Smooth}(K)=
 \min\left\{
 \CRad_D(\widehat\gamma)\ \middle|\
 P\in\Pol_N^{\Lat}(K),\ \gamma\in\Smooth(P)
 \right\},
\]
and
\[
 \CRad_{D,N}^{\Lat,\Move,\Smooth,\seed}(K;S)=
 \min\left\{
 \CRad_D(\widehat\gamma)\ \middle|\
 P\in G_N^{\Lat,\Move}(K;S),\ \gamma\in\Smooth(P)
 \right\}.
\]
\end{definition}

\begin{theorem}[Finite computability of smoothed profiles]\label{thm:smoothed-finite-computability}
Fix $K$, $N$, a finite seed-generated graph $G_N^{\Lat,\Move}(K;S)$, and a finite
smoothing-perturbation scheme $\Smooth$.  Then
\[
 \rho_{p,N}^{\Lat,\Move,\Smooth,\seed}(K;S),
 \qquad
 \CRad_{D,N}^{\Lat,\Move,\Smooth,\seed}(K;S)
\]
are finite computable quantities up to numerical precision.  If the full finite
set $\Pol_N^{\Lat}(K)$ is enumerated, the corresponding global smoothed quantities
are also computable.
\end{theorem}

\begin{proof}
The finite graph has finitely many vertices, and each vertex has finitely many
smoothed representatives.  Thus the feasible set is finite.  Evaluating the
chosen geometric functional on each element and taking a finite minimum gives the
claim.
\end{proof}

\begin{remark}[Why perturbation is not cosmetic]
The perturbation step is not merely a numerical trick.  It separates the
geometric invariant being approximated from special symmetries and angular
constraints of the simple cubic lattice.  A $90^\circ$ corner in a lattice polygon
imposes a fixed local curvature scale, while a smoothed $C^{1,1}$ perturbation
allows the experiment to probe nearby continuous representatives of the same
knot type.
\end{remark}

%==================================================
\section{Length windows and discrete-to-continuous comparison}\label{sec:length-windows}
%==================================================

Let $\lambda\ge1$.  The natural lattice analogue of a ropelength window is the
length cap
\[
 N\le \lambda n_{\Lat}(K).
\]
Since $N$ is integral and BFACF moves in the simple cubic lattice usually change
length by even numbers, the effective levels are discrete.

\begin{definition}[Windowed smoothed lattice profiles]
For the full lattice profile, define
\[
 \rho_{p,\lambda}^{\Lat,\Smooth}(K)=
 \min_{N\in\mathbb Z,\,N\le \lambda n_{\Lat}(K)}
 \rho_{p,N}^{\Lat,\Smooth}(K)
\]
and
\[
 \CRad_{D,\lambda}^{\Lat,\Smooth}(K)=
 \min_{N\in\mathbb Z,\,N\le \lambda n_{\Lat}(K)}
 \CRad_{D,N}^{\Lat,\Smooth}(K).
\]
For a seed-generated computation from seeds of length $N_0$, the minimum is
taken only over admissible levels
\[
 N\in\mathcal N_{\Move}(N_0),\qquad N\le \lambda N_0,
\]
and the full level-$N$ set is replaced by $G_N^{\Lat,\Move}(K;S)$.  In the
simple-cubic BFACF pilot computations this means $N=N_0,N_0+2,N_0+4,\ldots$.
\end{definition}

\begin{problem}[Thickness-controlled lattice-smoothing approximation]
Let $K$ be a knot type and let $\lambda\ge1$.  Suppose $\Lat_h=h\ZZ^3$ is a
sequence of cubic lattices with mesh $h\to0$.  Instead of using the dimensionless
cap $N\le \lambda n_{\Lat_h}(K)$, choose caps scaled to the continuous ropelength
window, for example
\[
 N_h(\lambda)=\left\lfloor \frac{\lambda\Rop(K)}{h}\right\rfloor
\]
with the parity adjustment required by the move system.  Find hypotheses on
smoothing schemes $\Smooth_h$, thickness certificates, and lattice approximation
under which
\[
 \rho_{p,N_h(\lambda)}^{\Lat_h,\Smooth_h}(K)
 \longrightarrow \rho_{p,\lambda}^{\operatorname{rop}}(K)
\]
and similarly for $\CRad_D$.
\end{problem}

\begin{remark}
This is a guiding approximation problem, not a result of the present paper.  The
unconditional statements below concern finite computability and monotonicity at a
fixed lattice and fixed search protocol.  The limit problem requires independent
control of lattice approximation, smoothing, thickness certificates, and
optimization error.
\end{remark}

\begin{problem}[Effective error bounds]
Find explicit functions $E_\rho(h,N,\eps,r)$ and $E_D(h,N,\eps,r)$ such that
smoothed lattice profiles differ from the continuous ropelength-windowed profiles
by at most these errors under stated reach and thickness hypotheses.
\end{problem}

%==================================================
\section{Numerical evaluation of the functionals}
%==================================================

This section records the computational formulas used in the experiments.

\subsection{$p=2$}
For any curve $\gamma$ of length $L$ with arclength centroid
\[
 \bar x=\frac1L\int_\gamma x\,ds,
\]
one has
\[
 \sD_2(\gamma)^2
 =\frac1{L^2}\int_\gamma\int_\gamma |x-y|^2\,ds_x\,ds_y
 =2\Rg(\gamma)^2,
\]
where
\[
 \Rg(\gamma)^2=\frac1L\int_\gamma |x-\bar x|^2\,ds.
\]
Thus $p=2$ is especially stable numerically.

\subsection{$p=\infty$}
For $p=\infty$,
\[
 \sD_\infty(\gamma)=\diam(\gamma).
\]
For a polygonal curve, the diameter is attained by vertices.  For a smoothed
curve, the diameter can be approximated by sufficiently fine arclength sampling,
then certified by interval or subdivision bounds if needed.

\subsection{General $p$}
For a polygonal curve, the double integral decomposes into a sum over pairs of
edges:
\[
 \int_P\int_P |x-y|^p\,ds_x\,ds_y
 =\sum_{e,f}\int_e\int_f |x-y|^p\,ds_x\,ds_y.
\]
For $p$ outside the stable closed-form range, numerical quadrature can be applied
edge-pairwise.  For $-1<p<0$, special care is needed near the diagonal pairs
$e=f$ because of the integrable singularity.

\subsection{Thickness after smoothing}
The smoothed representative $\gamma$ is required to be $C^{1,1}$.  Its thickness
is
\[
 \Thi(\gamma)=\min\{\operatorname{MinRad}(\gamma),\tfrac12\operatorname{dcsd}(\gamma)\},
\]
where $\operatorname{MinRad}$ is the minimum radius of curvature and
$\operatorname{dcsd}$ denotes doubly critical self-distance.  Numerically, both
quantities should be computed with certified lower bounds when the result is to
be used as more than exploratory evidence.

%==================================================
\section{Preliminary numerical sanity checks}\label{sec:raw-pilots}
%==================================================

This section supplies preliminary numerical sanity checks of the definitions before the full
smoothed BFACF search is run.  The values in this section are \emph{raw lattice}
values: no perturbation, corner smoothing, or thickness renormalization is applied.
Thus they are not yet the final smoothed profiles.  They demonstrate that the
functionals are directly computable on explicit lattice representatives and that
increasing the length level can change the best observed density.

For a unit-edge polygon $P$, the $p=2$ value was computed from
\[
 \sD_2(P)^2=\frac{2}{\ell(P)}\int_P |x-\bar x|^2\,ds,
 \qquad
 \bar x=\frac{1}{\ell(P)}\int_P x\,ds.
\]
For the polygonal segment from $a$ to $a+d$, where $|d|=1$, the required exact
integrals are
\[
 \int_0^1 (a+td)\,dt=a+\frac{d}{2},
 \qquad
 \int_0^1 |a+td|^2\,dt=|a|^2+a\cdot d+\frac13.
\]
The diameter was computed as the maximum distance between vertices.  In the raw
simple-cubic rows we use the labelled lattice-thickness convention
$\Thi_{\rm poly}=1/2$, so that $\CRad_D^{\SC}(P)=2D(P)$.  These raw compression
values are therefore convention-dependent placeholders for the later smoothed,
certified-thickness computation.

The 24-step trefoil seed used below is the cubic-lattice coordinate model
available from KnotPlot~\cite{KnotPlotCoords}; Diao proved that no nontrivial
simple-cubic lattice knot has fewer than 24 steps, and that the 24-step
nontrivial examples are trefoils~\cite{DiaoMinimal,DiaoSmallest}.
The length-26 trefoil row is obtained from this seed by replacing the edge
\[
 (0,2,1)\to (-1,2,1)
\]
with the plaquette detour
\[
 (0,2,1)\to (0,3,1)\to (-1,3,1)\to (-1,2,1).
\]
This is a single embedded local length-increasing detour and hence remains in the
same knot type.  It is not an exhaustive level-26 computation.

\begin{table}[ht]
\centering
\small
\begin{tabular}{llrrrrrr}
\toprule
$K$ & representative & $\ell$ & $\sD_2^2$ & $\sD_2$ & $\sD_\infty$ & $\rho_2$ & $\rho_\infty$ \\
\midrule
$0_1$ & unit square & 4 & $2/3$ & 0.8165 & 1.4142 & 4.8990 & 2.8284 \\
$0_1$ & $1\times2$ rectangle & 6 & $3/2$ & 1.2247 & 2.2361 & 4.8990 & 2.6833 \\
$3_1$ & 24-step seed & 24 & $643/144$ & 2.1131 & 4.1231 & 11.3576 & 5.8209 \\
$3_1$ & one $+2$ plaquette detour & 26 & $5359/1014$ & 2.2989 & 4.8990 & 11.3097 & 5.3072 \\
\bottomrule
\end{tabular}
\caption{Pilot raw-lattice density values.  These entries are candidate values,
not exhaustive minima.  The length-26 row gives an upper bound for the level-26
raw-lattice profile and shows that the length filtration can lower the observed
$\rho_2$ and $\rho_\infty$ values.}
\label{tab:pilot-density}
\end{table}

\begin{table}[ht]
\centering
\small
\begin{tabular}{llrrrr}
\toprule
$K$ & representative & $\ell$ & $\CRad_2^{\SC}=2\sD_2$ & $\CRad_\infty^{\SC}=2\sD_\infty$ & status \\
\midrule
$0_1$ & unit square & 4 & 1.6330 & 2.8284 & exact raw value \\
$0_1$ & $1\times2$ rectangle & 6 & 2.4495 & 4.4721 & exact raw value \\
$3_1$ & 24-step seed & 24 & 4.2262 & 8.2462 & seed value \\
$3_1$ & one $+2$ plaquette detour & 26 & 4.5978 & 9.7980 & non-exhaustive candidate \\
\bottomrule
\end{tabular}
\caption{Pilot raw-lattice compression-radius values under the standard
$\Thi_{\rm poly}=1/2$ convention.  The one-detour row worsens the raw compression
values even though it improves the displayed density values in
Table~\ref{tab:pilot-density}; this illustrates that density and compactness are
competing objectives.  Smoothed profiles will replace these raw values by
certified thickness-normalized $C^{1,1}$ values.}
\label{tab:pilot-compression}
\end{table}

\begin{remark}
The density and compression objectives need not improve simultaneously.  In
Table~\ref{tab:pilot-density}, the length-26 detour lowers the observed
$\rho_2$ and $\rho_\infty$ relative to the displayed 24-step seed, while
Table~\ref{tab:pilot-compression} shows that the raw compression radii increase.
This supports treating density and compression as complementary profiles on the
same filtered graph rather than as interchangeable measurements.
\end{remark}

\subsection{One complete smoothing sanity check: the rounded square}

As a minimal check of the smoothing-normalization pipeline, take the unit square
unknot and replace each corner by a quarter circle of radius $r=1/4$.  The
resulting planar $C^{1,1}$ curve consists of four line segments and four circular
arcs.  Its length is
\[
 L=4(1-2r)+2\pi r=2+\frac{\pi}{2}\approx 3.5708.
\]
For this convex rounded square, the curvature radius is $r$ on the arcs and the
opposite-side separation is larger than $2r$, so the elementary certificate gives
\[
 \Thi(\gamma)\ge r=1/4.
\]
The following values are obtained by direct line-and-arc quadrature.  They are
included only to show the full pipeline ``smoothing -- thickness certificate --
normalization -- evaluation'' in a completely transparent example.

\begin{table}[ht]
\centering
\small
\begin{tabular}{lrrrrrr}
\toprule
example & $L$ & $\tau_{\rm cert}$ & $\sD_2$ & $\sD_\infty$ & $\rho_2$ & $\CRad_2$ \\
\midrule
rounded square, $r=1/4$ & 3.5708 & 0.2500 & 0.7800 & 1.2071 & 4.5780 & 3.1199 \\
\bottomrule
\end{tabular}
\caption{A complete smoothing sanity check for the unknot.  The compression value
uses the certified lower bound $\tau_{\rm cert}=1/4$.  In general this gives the
conservative certified upper value $\sD_2/\tau_{\rm cert}$ for the true
compression radius; in this convex rounded-square example the certificate is
sharp, since the actual thickness is $r=1/4$.}
\label{tab:rounded-square}
\end{table}

This example is not intended to optimize the unknot.  Its role is to demonstrate
that the smoothed framework is not merely formal: for line-and-arc curves, the
quantities can be evaluated after a concrete thickness certificate has been
recorded.

\subsection{First length-filtration profile for the trefoil}

We next record a small length-filtration experiment.  Starting from the same
24-step trefoil seed, we generated candidates by applying only positive
BFACF-type $+2$ plaquette detours and retaining embedded polygons.  Thus the
computation below is deliberately restricted: it is neither an exhaustive BFACF
cap search nor an enumeration of all trefoil lattice polygons at the indicated
lengths.  Its purpose is to test whether the proposed functionals exhibit a
visible length-filtration effect already in a reproducible pilot search.

For each exact length level $N=24,26,28,30,32$, Table~\ref{tab:trefoil-exact-N}
records the number of generated candidates at that exact length and the best raw
values among those candidates.  Table~\ref{tab:trefoil-filtered-N} records the
corresponding filtered values, namely the best values among all generated
candidates of length at most $N$.

\begin{table}[ht]
\centering
\small
\begin{tabular}{rrrrrr}
\toprule
exact length $N$ & generated candidates & best $\rho_2$ & best $\rho_\infty$ & min $\CRad_2$ & min $\CRad_\infty$ \\
\midrule
24 & 1 & 11.3576 & 5.8209 & 4.2262 & 8.2462 \\
26 & 47 & 11.3097 & 5.3072 & 4.1827 & 8.2462 \\
28 & 1286 & 10.8903 & 4.8742 & 4.1886 & 8.2462 \\
30 & 27018 & 10.3115 & 4.5227 & 4.2153 & 8.2462 \\
32 & 485184 & 9.7147 & 4.2385 & 4.2420 & 8.2462 \\
\bottomrule
\end{tabular}
\caption{Exact-level raw-lattice pilot profile for the trefoil.  Candidates are
obtained from a fixed 24-step seed by positive $+2$ plaquette detours only.  The
values are therefore non-exhaustive candidate values.}
\label{tab:trefoil-exact-N}
\end{table}

\begin{table}[ht]
\centering
\small
\begin{tabular}{rrrrrr}
\toprule
cap $N$ & cumulative & best $\rho_2$ & best $\rho_\infty$ & best $\CRad_2$ & best $\CRad_\infty$ \\
\midrule
24 & 1 & 11.3576 & 5.8209 & 4.2262 & 8.2462 \\
26 & 48 & 11.3097 & 5.3072 & 4.1827 & 8.2462 \\
28 & 1334 & 10.8903 & 4.8742 & 4.1827 & 8.2462 \\
30 & 28352 & 10.3115 & 4.5227 & 4.1827 & 8.2462 \\
32 & 513536 & 9.7147 & 4.2385 & 4.1827 & 8.2462 \\
\bottomrule
\end{tabular}
\caption{Filtered raw-lattice pilot profile for the trefoil.  The entries are
minima over the generated candidates with length at most $N$.  The column
``cumulative'' gives the number of generated candidates with $\ell\le N$.
The monotone behavior follows from the nested feasible sets, while the numerical
improvement measures the effect of increasing the length cap within this
restricted search.}
\label{tab:trefoil-filtered-N}
\end{table}

The density profiles improve clearly over this range.  In the filtered profile,
$\rho_2$ decreases from $11.3576$ at $N=24$ to $9.7147$ at $N=32$, and
$\rho_\infty$ decreases from $5.8209$ to $4.2385$.  Thus the length filtration
reveals representatives that are more spread out, in the sense relevant to the
density functionals, than the minimal-length seed.  The extension from $N=30$ to
$N=32$ is especially useful as a stress test: the generated exact-level set grows
from $27018$ to $485184$ candidates, but the best density values continue to
improve noticeably.

The compression profiles behave differently.  The best observed $\CRad_2$ is
already achieved by length $26$ in this restricted search, and the best observed
$\CRad_\infty$ remains the seed value through $N=32$.  The exact-level values
also show that adding length does not automatically improve compression: the
best exact-level $\CRad_2$ increases mildly after $N=26$.  This is consistent
with the fact that lowering density usually favors spatially larger
representatives, whereas compression radius penalizes spatial size after
thickness normalization.  The two families of functionals therefore give
complementary, rather than redundant, measurements on the same filtered move
graph.

\subsection{Global minimal-layer merge hierarchies underlying the certificates}
\label{subsec:global-merge-hierarchies}

The path-level computations below should now be read inside the complete
minimal-layer BFACF hierarchies determined in the companion paper
\cite{OzawaDiscrete}.  For $4_1$, the $152$ minimal classes form four type-$0$
components of sizes $58,58,18,18$ at $N=30$, and all four merge at $N=32$.
For $6_3$, the $148$ minimal classes form twelve type-$0$ components of sizes
$39,39,10,10,9,9,7,7,6,6,3,3$ at $N=40$.  We use the companion paper's
labelling
\[
A_{39},A_{10},A_9,A_7,A_6,A_3,
\qquad
B_{39},B_{10},B_9,B_7,B_6,B_3,
\]
where reflection exchanges $A_r$ and $B_r$.  At $N=42$ the six $A_r$
components form one branch $A$ and the six $B_r$ components form its mirror
branch $B$; each branch contains $74$ minimal classes and $12337$ bounded
states, and the two branches are disjoint.  They merge at $N=44$.  Thus the
length-only merge hierarchies are
\[
4\longrightarrow1 \quad (4_1),
\qquad
12\longrightarrow2\longrightarrow1 \quad (6_3).
\]
The certificates studied in the next two subsections connect opposite mirror
branches and therefore realize the largest birth-level barrier in each example.
This distinction is useful for the present paper: the length filtration is known
globally on the minimal layer, and we can therefore attach geometric data to all
of its leaves and internal nodes.  The next subsection carries out this
calculation before returning to selected path certificates.

\begin{table}[ht]
\centering
\small
\begin{tabular}{lcp{7.0cm}c}
\toprule
knot & minimal length & birth-level partition and intermediate merge & final merge\\
\midrule
$4_1$ & 30 & $58+58+18+18$; all four components persist to the next admissible level & $N=32$\\
$6_3$ & 40 & $39+39+10+10+9+9+7+7+6+6+3+3$; $12\to2$ at $N=42$ & $N=44$\\
\bottomrule
\end{tabular}
\caption{Complete minimal-layer BFACF component hierarchies from the companion
paper.  Reflections are retained in the quotient.}
\label{tab:global-minimal-merge-hierarchies}
\end{table}

\subsection{Exact raw geometric profiles on the merge trees}
\label{subsec:tree-geometric-profiles}

The complete merge trees allow the geometric vertex functionals to be evaluated
componentwise rather than only along selected paths.  If $C$ is a component
represented by a node of a minimal-layer merge tree, define its
\emph{minimal support} by
\[
 \mathcal M(C)=C\cap \Pol_{n_{\SC}(K)}^{\SC}(K).
\]
When the whole bounded component $C\subset G_N^{\SC,\BFACF}(K)$ has been
exhaustively enumerated, we also use its \emph{bounded profile}.  For a raw
vertex functional $F\in\{\rho_2,\rho_\infty,\CRad_2^{\SC},\CRad_\infty^{\SC}\}$
we record the minimum, maximum, mean, and median on these finite sets.  The
complete state-level values and the four summary statistics are included in the
supplementary CSV files; the tables below display the minima, which are the
quantities most directly comparable with the filtered optimization profiles.

\begin{proposition}[Exact raw profiles on the computed merge-tree nodes]
\label{prop:exact-tree-profiles}
Under the mirror-sensitive quotient and the labelled raw convention
$\Thi_{\rm poly}=1/2$, the following values are exact for the stated finite
sets.
\begin{enumerate}[label=(\roman*)]
\item For $4_1$, all $152$ minimal states and the entire $19618$-state bounded
component at $N=32$ are evaluated.
\item For $6_3$, all $148$ minimal states and both entire $12337$-state bounded
components $A$ and $B$ at $N=42$ are evaluated.
\item Reflection preserves all four functionals.  Consequently each pair
$A_r,B_r$ has the same complete distribution, and the two level-$42$ branches
$A$ and $B$ have the same complete distribution, although $A\cap B=\varnothing$
at $N=42$.
\item The minimal-support profile of the final $6_3$ root is exact, since its
support is the union of all $148$ minimal states.  No exhaustive bounded-state
profile is claimed for the full $N=44$ root.
\end{enumerate}
\end{proposition}

\begin{proof}
The state sets are finite and were enumerated with the same BFACF moves and
canonicalization as in the companion merge-tree computation.  The $4_1$
level-$32$ enumeration contains $19618$ states, including all $152$ minimal
states.  For $6_3$, the two disjoint level-$42$ enumerations contain $12337$
states each and together contain all $148$ minimal states.  The four raw
functionals are then evaluated directly on every canonical polygon.  Finally,
reflection preserves arclength measure and Euclidean distance, hence preserves
$\sD_2$, diameter, and therefore all four displayed functionals.  This proves
the equality of the full mirror-paired distributions.
\end{proof}

For $4_1$, the two mirror pairs already exhibit a density--compression trade-off
at the birth level.  The $58$-state components have smaller density minima, while
the $18$-state components have smaller compression minima.  After the four
components merge at $N=32$, new length-$32$ states improve both density minima,
but neither compression minimum improves.

\begin{table}[ht]
\centering
\small
\begin{tabular}{lrrrrr}
\toprule
$4_1$ node/set & states & $\min\rho_2$ & $\min\rho_\infty$ & $\min\CRad_2^{\SC}$ & $\min\CRad_\infty^{\SC}$\\
\midrule
$C_1$ or $C_2$ minimal support & 58 & 12.4259 & 6.3960 & 4.4930 & 8.9443\\
$C_3$ or $C_4$ minimal support & 18 & 13.3191 & 7.2761 & 4.3604 & 7.4833\\
root minimal support & 152 & 12.4259 & 6.3960 & 4.3604 & 7.4833\\
root bounded component, $N=32$ & 19618 & 11.8473 & 5.8424 & 4.3604 & 7.4833\\
\bottomrule
\end{tabular}
\caption{Exact raw minima on the complete $4_1$ merge tree.  The mirror pairs
$C_1,C_2$ and $C_3,C_4$ have identical complete distributions.  Full
min/max/mean/median summaries are supplied in
\texttt{tree\_geometric\_profile\_summary.csv}.}
\label{tab:fig8-tree-geometric-minima}
\end{table}

For $6_3$, each row below represents both members of a mirror pair.  Different
birth components optimize different functionals: for example $A_7/B_7$ gives
the smallest birth-level $\rho_2$, $A_9/B_9$ the smallest birth-level
$\rho_\infty$, and $A_6/B_6$ the smallest birth-level $\CRad_2^{\SC}$.  Thus no
single birth component is geometrically optimal for all four observables.

\begin{table}[ht]
\centering
\small
\begin{tabular}{lrrrrr}
\toprule
$6_3$ node/set & states & $\min\rho_2$ & $\min\rho_\infty$ & $\min\CRad_2^{\SC}$ & $\min\CRad_\infty^{\SC}$\\
\midrule
$A_{39}$ or $B_{39}$ & 39 & 15.6243 & 7.8446 & 4.9048 & 10.0000\\
$A_{10}$ or $B_{10}$ & 10 & 15.7694 & 7.8446 & 4.9363 & 10.1980\\
$A_9$ or $B_9$ & 9 & 15.2841 & 7.3030 & 5.0186 & 10.3923\\
$A_7$ or $B_7$ & 7 & 14.9132 & 7.4278 & 5.1494 & 10.1980\\
$A_6$ or $B_6$ & 6 & 16.0590 & 7.8446 & 4.8597 & 10.0000\\
$A_3$ or $B_3$ & 3 & 15.5651 & 7.8446 & 5.0790 & 10.1980\\
$A$ or $B$: minimal support, $N=42$ & 74 & 14.9132 & 7.3030 & 4.8597 & 10.0000\\
$A$ or $B$: bounded component, $N=42$ & 12337 & 14.1660 & 6.5593 & 4.8464 & 9.7980\\
root minimal support, $N=44$ & 148 & 14.9132 & 7.3030 & 4.8597 & 10.0000\\
\bottomrule
\end{tabular}
\caption{Exact raw minima on the computed nodes of the complete $6_3$ merge
tree.  Mirror-related rows have identical full distributions.  The bounded
$N=44$ root has not been exhaustively enumerated, so only its exact minimal
support is listed.}
\label{tab:six3-tree-geometric-minima}
\end{table}

The level-$42$ merge therefore opens genuinely new geometric regimes: relative
to the $74$ minimal states in either branch, the bounded branch improves the
four minima
\[
(\rho_2,\rho_\infty,\CRad_2^{\SC},\CRad_\infty^{\SC})
=(14.9132,7.3030,4.8597,10.0000)
\]
to
\[
(14.1660,6.5593,4.8464,9.7980).
\]
At the same time, the two disconnected branches $A$ and $B$ remain completely
indistinguishable by the distributions of these reflection-invariant scalar
observables.  Thus the merge tree records reconfiguration information that is
not recoverable from density or compression alone.

\subsection{Raw profiles along the figure-eight merge certificate}
\label{subsec:fig8-merge-profiles}

We next record how the same raw density and compression quantities behave along
an explicit merge certificate coming from the companion lattice-filtered
move-graph computation~\cite{OzawaDiscrete}; the verified path data are archived
in~\cite{OzawaDiscreteData}.  The path connects a
supplied $30$-edge simple-cubic figure-eight seed $\omega$ to its reflected
mirror seed $\omega!$ under the length bound $N=32$.  It has $21$ states and
$20$ BFACF moves and passes through a $32$-edge connecting state $\eta$.  The
seed $\omega$ reaches $\eta$ in $5$ BFACF moves, while the mirror seed
$\omega!$ reaches the same state in $15$ BFACF moves.  As in the preceding raw
lattice tables, we use
\[
 \rho_2(P)=\frac{\ell(P)}{\sD_2(P)},\qquad
 \rho_\infty(P)=\frac{\ell(P)}{\diam(P)},
\]
and the labelled simple-cubic convention
\[
 \CRad_2^{\SC}(P)=2\sD_2(P),\qquad
 \CRad_\infty^{\SC}(P)=2\diam(P).
\]

\begin{table}[ht]
\centering
\small
\begin{tabular}{llrrrrrrr}
\toprule
state & role & $\ell$ & $\sD_2$ & $\sD_\infty$ & $\rho_2$ & $\rho_\infty$ & $\CRad_2^{\SC}$ & $\CRad_\infty^{\SC}$ \\
\midrule
$s_0$ & seed $\omega$ & 30 & 2.3338 & 4.5826 & 12.8545 & 6.5465 & 4.6676 & 9.1652 \\
$s_5$ & connecting state $\eta$ & 32 & 2.3532 & 4.5826 & 13.5982 & 6.9830 & 4.7065 & 9.1652 \\
$s_{20}$ & mirror seed $\omega!$ & 30 & 2.3338 & 4.5826 & 12.8545 & 6.5465 & 4.6676 & 9.1652 \\
\bottomrule
\end{tabular}
\caption{Raw density and compression-radius values at the endpoints and the
connecting state of the extracted $N=32$ BFACF merge certificate for the supplied
figure-eight seed and its reflected mirror seed.  These are raw lattice values,
not smoothed thickness-normalized values.}
\label{tab:fig8-merge-key-values}
\end{table}

The connecting state has the same diameter spread as the two seeds, but a
slightly larger $\sD_2$ and two additional lattice edges.  Consequently, along
the passage through $\eta$ one has
\[
 \rho_2: 12.8545\longrightarrow 13.5982,
 \qquad
 \rho_\infty: 6.5465\longrightarrow 6.9830,
\]
and
\[
 \CRad_2^{\SC}: 4.6676\longrightarrow 4.7065,
 \qquad
 \CRad_\infty^{\SC}: 9.1652\longrightarrow 9.1652.
\]
Thus this particular merge certificate temporarily worsens the displayed raw
density values, while leaving the diameter-based compression radius unchanged at
$\eta$ and slightly increasing the $p=2$ compression radius.

The full path is more informative than the three distinguished states alone.  Its
edge-number sequence is
\[
\begin{gathered}
30,32,32,32,32,32,32,32,32,32,30,\\
32,32,32,32,32,32,30,30,30,30.
\end{gathered}
\]
At the intermediate state $s_{10}$, for example, the compression values improve
to
\[
 \CRad_2^{\SC}=4.4969,
 \qquad
 \CRad_\infty^{\SC}=8.2462,
\]
while the density values worsen to
\[
 \rho_2=13.3425,
 \qquad
 \rho_\infty=7.2761.
\]
This is a concrete path-level instance of the same phenomenon seen in the
trefoil pilot search: minimizing density and minimizing compression radius are
competing objectives on a filtered move graph.

\begin{figure}[ht]
\centering
\includegraphics[width=.78\textwidth]{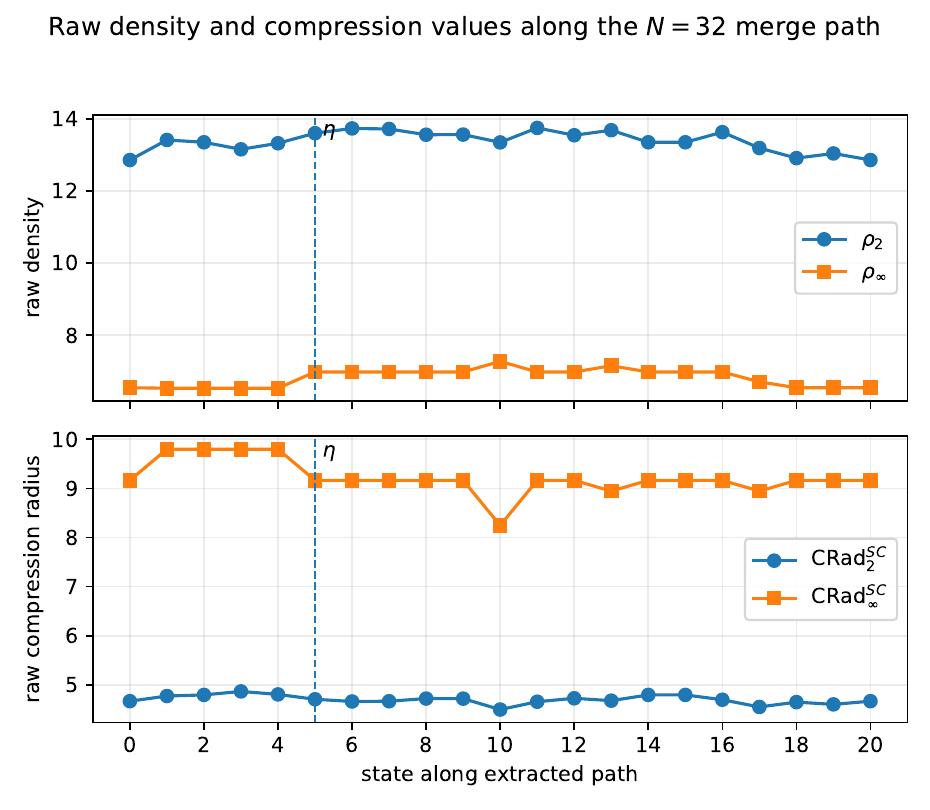}
\caption{Raw density and compression values along the extracted $21$-state
figure-eight BFACF merge path at length bound $N=32$.  The dashed vertical line
marks the connecting state $\eta=s_5$.  The plotted data are raw lattice values
under the labelled convention $\Thi_{\rm poly}=1/2$.}
\label{fig:fig8-merge-density-compression}
\end{figure}

In the raw simple-cubic convention used here, the formal product satisfies
\[
 \rho_D(P)\CRad_D^{\SC}(P)=2\ell(P)
\]
for both $D=\sD_2$ and $D=\diam$.  Hence the merge path raises the raw ropelength
proxy from $60$ at the $30$-edge seeds to $64$ at the $32$-edge states and then
returns to $60$.  This is the density--compression interpretation of one extremal path in the
complete $4_1$ hierarchy: its endpoints lie in distinct birth-level components,
and every distinct pair of birth-level components first becomes connected under
the cap $N=32$.  The full numerical table for the $21$ states is
included in the supplementary data as
\path{figure8_merge_path_32_density_compression.csv}.

\subsection{Raw profiles along the $6_3$ merge certificate}
\label{subsec:six3-merge-profiles}

We now apply the same raw functionals to a verified BFACF certificate realizing
the maximal barrier in the complete $6_3$ hierarchy from~\cite{OzawaDiscrete}.
Let $\sigma$ be the supplied $40$-edge minimal simple-cubic seed and let
$\sigma!$ be its reflected mirror.  With the labelling above, $\sigma\in A_7$
and $\sigma!\in B_7$.  By the global computation, at $N=42$ these components
lie in the opposite branches $A$ and $B=\mu(A)$, each containing $74$ minimal
classes and $12337$ bounded states; the two branches are disjoint.  At $N=44$,
a verified explicit path connects them.  The path has $153$ states and
$152$ BFACF moves, reaches length $44$ but never exceeds it, and meets at a
$42$-edge state $\eta=s_{116}$.  The seed branch has $116$ moves and the mirror
branch has $36$ moves.  Hence this path realizes the global maximal excess
barrier $4$ between the two level-$42$ branches.

\begin{table}[ht]
\centering
\small
\begin{tabular}{llrrrrrrr}
\toprule
state & role & $\ell$ & $\sD_2$ & $\sD_\infty$ & $\rho_2$ & $\rho_\infty$ & $\CRad_2^{\SC}$ & $\CRad_\infty^{\SC}$ \\
\midrule
$s_0$ & seed $\sigma$ & 40 & 2.6171 & 5.0990 & 15.2841 & 7.8446 & 5.2342 & 10.1980 \\
$s_{116}$ & meeting state $\eta$ & 42 & 2.5384 & 5.0990 & 16.5461 & 8.2369 & 5.0767 & 10.1980 \\
$s_{152}$ & mirror seed $\sigma!$ & 40 & 2.6171 & 5.0990 & 15.2841 & 7.8446 & 5.2342 & 10.1980 \\
\bottomrule
\end{tabular}
\caption{Raw density and compression-radius values at the endpoints and meeting
state of the extracted $N=44$ BFACF merge certificate for the supplied $6_3$
seed and its reflected mirror.  The values use the labelled convention
$\Thi_{\rm poly}=1/2$ and are not smoothed thickness-normalized values.}
\label{tab:six3-merge-key-values}
\end{table}

The $6_3$ meeting state behaves differently from the figure-eight meeting state.
Its diameter is unchanged, but its mean chord spread decreases.  Relative to the
seed, one therefore has
\[
 \rho_2:15.2841\longrightarrow16.5461,
 \qquad
 \rho_\infty:7.8446\longrightarrow8.2369,
\]
while
\[
 \CRad_2^{\SC}:5.2342\longrightarrow5.0767,
 \qquad
 \CRad_\infty^{\SC}:10.1980\longrightarrow10.1980.
\]
Thus the meeting state improves the raw $p=2$ compression radius even though it
worsens both density values.  This gives a particularly direct path-level
example of the competition between spreading and compactness.

The full certificate contains further trade-offs.  The smallest value of
$\rho_2$ along the path is $15.0260$ at $s_3$, whereas the smallest value of
$\rho_\infty$ is $6.8716$ at $s_{52}$.  The latter state has
$\CRad_\infty^{\SC}=12.8062$, the largest diameter-based compression value on
the path.  Conversely, the minimum of $\CRad_2^{\SC}$ is $5.0767$ at the
meeting state $s_{116}$.  No single state simultaneously minimizes all four
quantities.

\begin{figure}[ht]
\centering
\includegraphics[width=.88\textwidth]{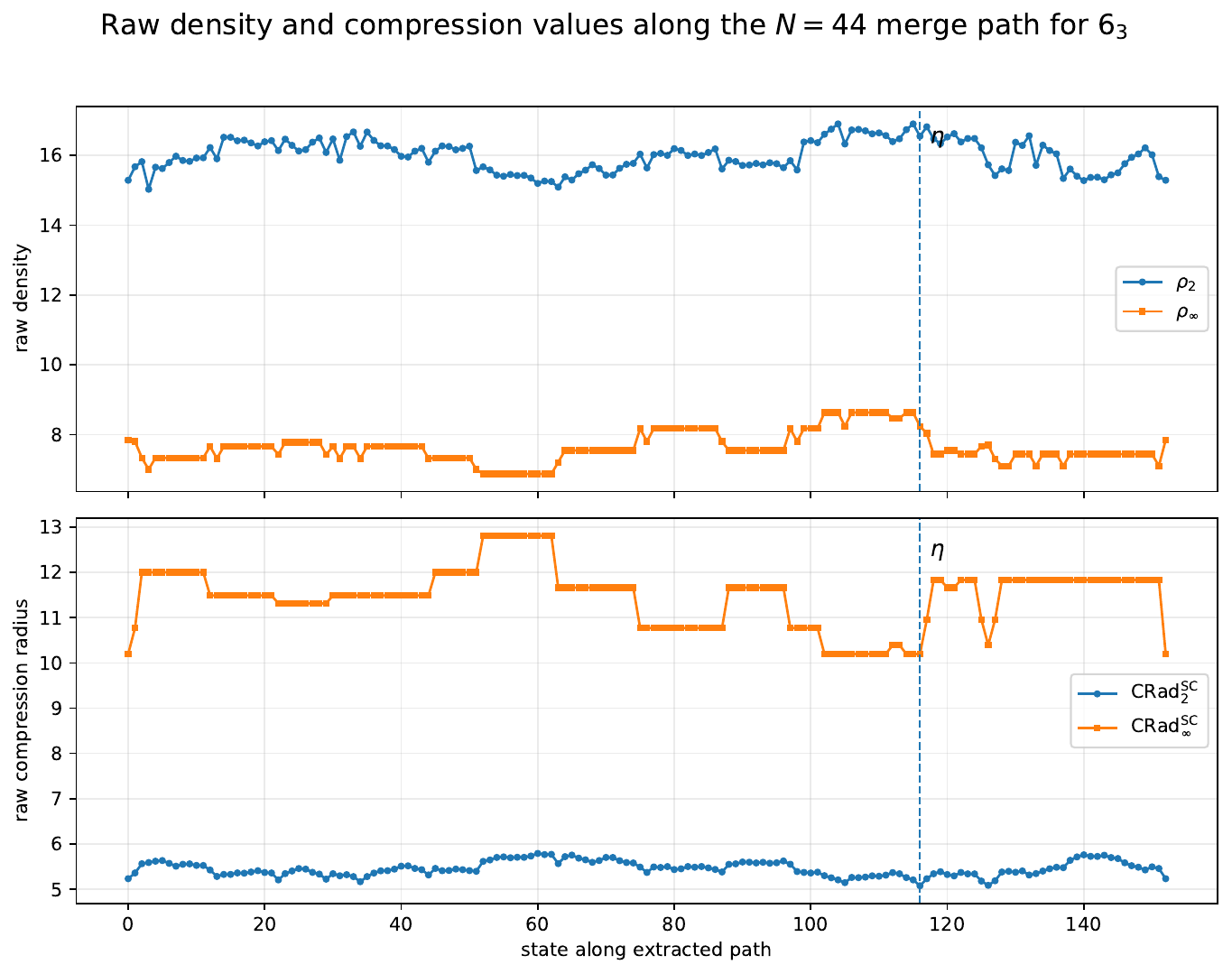}
\caption{Raw density and compression values along the extracted $153$-state
$6_3$ BFACF merge path at length bound $N=44$.  The dashed vertical line marks
the meeting state $\eta=s_{116}$.  The plotted data are raw lattice values under
the labelled convention $\Thi_{\rm poly}=1/2$.}
\label{fig:six3-merge-density-compression}
\end{figure}

For the raw cubic-lattice convention, the product identity again gives
\[
 \rho_D(P)\CRad_D^{\SC}(P)=2\ell(P).
\]
The raw ropelength proxy therefore takes the values $80$, $84$, and $88$ on
states of lengths $40$, $42$, and $44$, respectively.  The full derived table is
included in the revision source package as
\path{six3_merge_path_44_density_compression.csv}.

\begin{table}[ht]
\centering
\begin{tabular}{lrrrrr}
\toprule
knot & seed length & merge cap & mirror barrier & certificate & meeting length \\
\midrule
$4_1$ & 30 & 32 & 2 & 21 states, 20 moves & 32 \\
$6_3$ & 40 & 44 & 4 & 153 states, 152 moves & 42 \\
\bottomrule
\end{tabular}
\caption{Comparison of the two verified BFACF mirror certificates used for the
raw density and compression analysis.  In each knot these paths realize the
maximal excess barrier in the complete minimal-layer BFACF hierarchy; the
geometric values themselves remain path-specific.}
\label{tab:amphichiral-merge-comparison}
\end{table}

\subsection{Expected large-cap behavior and degeneracy}
\label{subsec:large-N-behavior}

The filtered columns in Table~\ref{tab:trefoil-filtered-N} are monotone by
construction: enlarging the length cap only enlarges the feasible set.  The
large-$N$ interpretation, however, is different for density and compression.
The observed behavior should not be read as convergence to a nondegenerate
continuous invariant.  Rather, it is evidence for why a ropelength window is
needed.

This is the raw-lattice counterpart of the continuous distinction between
unconstrained density degeneration~\cite{OzawaPDensity} and compactness-type
compression optimization~\cite{OzawaCompression}.

For the diameter density,
\[
 \rho_\infty(P)=\frac{\ell(P)}{\diam(P)},
\]
there is a universal lower bound
\[
 \rho_\infty(P)\ge 2,
\]
because any closed curve of length $L$ has diameter at most $L/2$.  Conversely,
one may localize the knotted part in a bounded region and put most of the length
into a long thin unknotted loop or stadium.  In such a family,
\[\diam(P)\sim \ell(P)/2,\]
and hence the unrestricted raw density is expected to approach the degenerate
limit
\[
 \rho_{\infty,N}(K)\longrightarrow 2
\]
as the cap tends to infinity, at least along sufficiently flexible families of
lattice representatives.  The trefoil values
\[
 5.8209,\quad 5.3072,\quad 4.8742,\quad 4.5227,\quad 4.2385
\]
are consistent with the initial stage of this degeneration.

For $p=2$ the expected limiting value depends on how much geometric freedom is
allowed.  In a fixed cubic lattice, a long thin out-and-back rectangular or
stadium-like degeneration behaves like a doubled interval.  If the doubled
interval has total length $L$, then
\[
 \sD_2^2\sim \frac{L^2}{24},
 \qquad
 \rho_2\sim \sqrt{24}=2\sqrt6\approx 4.8990.
\]
Thus the raw cubic-lattice profile may be expected to drift toward a value near
$\sqrt{24}$ in such restricted degenerating families.  By contrast, if smoothing
and mesh refinement allow the long unknotted part to approximate a round circle,
then a circle of length $L$ has
\[
 \sD_2=\frac{L}{\sqrt2\pi},
 \qquad
 \rho_2=\sqrt2\pi\approx 4.4429.
\]
Accordingly, the smoothed or mesh-refined unrestricted profile may have a lower
degenerate target than the fixed cubic-lattice raw profile.  These values are
not proposed as new knot invariants; they are benchmark degeneracies for
unwindowed density.

Compression profiles have the opposite character.  In the raw simple-cubic
normalization used in the tables,
\[
 \CRad_\infty(P)=2\diam(P),
 \qquad
 \CRad_2(P)=2\sD_2(P).
\]
Lowering density is usually achieved by making the representative spatially
larger, whereas lowering compression radius favors compact representatives.  It
is therefore expected that the filtered compression profiles stabilize at finite
positive values associated with compact lattice representatives.  In the present
restricted trefoil search, the best observed value of $\CRad_2$ is already
attained at length $26$, while the best observed $\CRad_\infty$ remains the
24-step seed value throughout $N\le32$.  A full or sampled BFACF search may lower
these values, but it should not exhibit the same degeneration to the universal
lower bounds seen by the density profiles.

In summary, the pilot data suggest the following qualitative picture:
\[
\begin{array}{c|c|c}
\text{quantity} & \text{observed cap behavior} & \text{expected unrestricted behavior}\\
\hline
\rho_\infty & \text{decreases} & \text{degenerates toward }2\\
\rho_2 & \text{decreases} & \text{toward long-loop benchmarks such as }
\sqrt{24}\text{ or }\sqrt2\pi\\
\CRad_\infty & \text{currently stable} & \text{finite compactness optimum}\\
\CRad_2 & \text{improves early, then stable} & \text{finite compactness optimum}
\end{array}
\]
This distinction is one of the main reasons for retaining the ropelength-windowed
formulation in the definitions.  Without such a window, the density profiles
mainly detect how efficiently a knot can be hidden inside a long nearly
unknotted loop; the compression profiles instead measure compactness under the
chosen thickness convention or certificate.

These tables should be read as a pilot computation.  In the terminology of the
experimental protocol, the labels are ``seed-generated'', ``positive-detour
restricted'', and ``raw lattice''.  The same table format will be used for the
smoothed-thickness-normalized computations once a certified smoothing and
thickness routine is incorporated.  Because the number of positive-detour
candidates grows rapidly, from $27018$ at exact length $30$ to $485184$ at exact
length $32$, computations beyond $N=32$ should either be explicitly labelled as
larger cap searches, or else use a pruned, beam-search, or BFACF-sampling label.

%==================================================
\section{Outlook: planned experimental protocol for small knots}
%==================================================

We now specify the planned computational protocol for the next experimental stage.  The protocol is designed to
be compatible with lattice-filtered BFACF computations and with later replacement
by exhaustive minimal-layer data.

\subsection{Target knot types}
The raw merge-certificate stage has now been carried out for supplied mirror
seed pairs of $4_1$ and $6_3$.  Both knots remain in the following target set
because the certified smoothed-thickness profiles have not yet been completed.
The first experimental set is
\[
\begin{gathered}
0_1,\quad 3_1,\quad 4_1,\quad 5_1,\quad 5_2,\\
6_1,\quad 6_2,\quad 6_3,\quad 7_1,\quad 7_2,\\
3_1\#3_1,\quad 3_1\#\overline{3_1}.
\end{gathered}
\]
Here $3_1\#3_1$ is the granny knot, while $3_1\#\overline{3_1}$ denotes the
square knot.  We do not include a separate reflected-seed diagnostic in the
present protocol, since the distance-based functionals studied here are not
designed to detect chirality.

\subsection{Length levels}
For each knot type $K$, begin from one or more supplied seeds at the known or
chosen birth level $N_0$.  Then compute at levels
\[
 N=N_0,
 N_0+2,
 N_0+4,
 \ldots,
 N_{\max}.
\]
For a window parameter $\lambda$, take
\[
 N_{\max}=\lfloor \lambda N_0\rfloor
\]
with parity adjusted to the move system.

\begin{table}[ht]
\centering
\begin{tabular}{lccc}
\toprule
$K$ & initial length $N_0$ & first levels & role in experiment \\
\midrule
$0_1$ & $4$ & $4,6,8,\ldots$ & calibration \\
$3_1$ & $24$ & $24,26,28,\ldots$ & first nontrivial prime \\
$4_1$ & $30$ & $30,32,34,\ldots$ & four-crossing prime \\
$5_1$ & $34$ & $34,36,38,\ldots$ & torus knot test \\
$5_2$ & $36$ & $36,38,40,\ldots$ & twist knot test \\
$6_1$ & $40$ & $40,42,44,\ldots$ & six-crossing prime \\
$6_2$ & $40$ & $40,42,44,\ldots$ & six-crossing prime \\
$6_3$ & $40$ & $40,42,44,\ldots$ & six-crossing prime \\
$7_1$ & $44$ & $44,46,48,\ldots$ & torus knot test \\
$7_2$ & $46$ & $46,48,50,\ldots$ & seven-crossing prime \\
$3_1\#3_1$ & $40$ & $40,42,44,\ldots$ & granny knot \\
$3_1\#\overline{3_1}$ & $40$ & $40,42,44,\ldots$ & square knot \\
\bottomrule
\end{tabular}
\caption{Initial target list for length-filtered smoothed density and compression
computations.  The values for $0_1$ and $3_1$ are classical simple-cubic minimal
lengths, with the trefoil minimum due to Diao~\cite{DiaoMinimal,DiaoSmallest};
for the remaining small prime knots and the listed connected sums we use the
standard simple-cubic minimum-length table of Janse van Rensburg and
Rechnitzer~\cite{MinimalKnots} as the intended birth levels.  In an actual data
release, each seed file should still be independently checked for length,
embeddedness, and knot type.  If a supplied seed is not certified minimal, the
corresponding $N_0$ should be interpreted as a supplied-seed birth level rather
than a proven lattice minimum.}
\label{tab:target-list}
\end{table}

\subsection{Per-level computation}
At each level $N$:
\begin{enumerate}[label=(\arabic*)]
\item Generate the reachable BFACF graph from the supplied seed set, with length
cap $N$.
\item Canonicalize vertices modulo orientation-preserving lattice isometries.
\item For each vertex $P$, generate a finite family $\Smooth(P)$ of smoothed
perturbations.
\item Compute $\Thi(\gamma)$, normalize to $\widehat\gamma$, and evaluate
$\rho_2$, $\rho_\infty$, $\CRad_2$, and $\CRad_\infty$.
\item Record the best value and the vertex/perturbation realizing it.
\end{enumerate}

\subsection{Output tables}
The main experimental output should have one table per functional, following the
format of Tables~\ref{tab:trefoil-exact-N} and~\ref{tab:trefoil-filtered-N}.
For example, for $\rho_2$:
\[
\begin{array}{c|cccc|c}
K & N_0 & N_0+2 & N_0+4 & N_0+6 & \text{best}\\\hline
3_1 & * & * & * & * & *\\
4_1 & * & * & * & * & *\\
\end{array}
\]
The entries should be labelled as one of the following:
\begin{itemize}
\item exhaustive at level $N$;
\item seed-generated at level $N$;
\item capped search at level $N$;
\item stochastic perturbation search at level $N$.
\end{itemize}
This labelling prevents experimental values from being mistaken for absolute
knot-type invariants.

%==================================================
\section{Outlook: connected-sum tests}
%==================================================

Connected sums provide a useful check on whether the profiles see composite
structure.  The pair $3_1\#3_1$ and $3_1\#\overline{3_1}$ will be used as a
standard composite test pair.  Since the functionals in this paper are based on
Euclidean distances, they should not be expected to detect chirality by
themselves.  The point of the comparison is instead to ask how density and
compression profiles behave under connected sum and whether low-density
representatives emerge only after the length cap is increased.

For each composite seed, the output should report the same labels as in the prime
case: whether the computation is exhaustive, seed-generated, capped, or
stochastic.  Any comparison between the granny and square knots must be read with
this search label in mind.

%==================================================
\section{Component profiles and merge scales}
%==================================================

The lattice-filtered move graph carries more information than a single global
minimum.  At level $N$, let $\Comp_N$ be the set of connected components of
$G_N^{\Lat,\Move}(K;S)$.  For $C\in\Comp_N$, define the componentwise smoothed
profile
\[
 \rho_p(C)=\min\{\rho_p(\widehat\gamma):P\in C,
       \gamma\in\Smooth(P)\},
\]
and similarly
\[
 \CRad_D(C)=\min\{\CRad_D(\widehat\gamma):P\in C,
       \gamma\in\Smooth(P)\}.
\]
The global seed-generated value is the minimum over components:
\[
 \rho_{p,N}^{\Lat,\Move,\Smooth,\seed}(K;S)
 =\min_{C\in\Comp_N}\rho_p(C).
\]

\begin{theorem}[Component monotonicity under merge]
Assume the seed-generated graphs are nested as $N$ increases.  Let
$C_1,\ldots,C_m$ be components of $G_N^{\Lat,\Move}(K;S)$ which lie in a single
component $C'$ of $G_{N'}^{\Lat,\Move}(K;S)$ for some $N'\ge N$.  Then
\[
 \rho_p(C')\le \min_{1\le j\le m}\rho_p(C_j),
 \qquad
 \CRad_D(C')\le \min_{1\le j\le m}\CRad_D(C_j).
\]
\end{theorem}

\begin{proof}
Each component $C_j$ is contained in the larger feasible component $C'$ at level
$N'$.  The finite set over which the minimum defining $\rho_p(C')$ is taken
therefore contains the feasible sets defining the $\rho_p(C_j)$.  The same
argument applies to $\CRad_D$.
\end{proof}

\begin{remark}
This theorem gives a simple but useful way to combine merge-scale data with
geometric optimization.  A merge event may lower the best density or compression
value available to a component, but it cannot force the componentwise minimum to
increase.  Thus one can record, for each persistent component, both its merge
scale and the first level at which a prescribed density or compression threshold
is achieved.
\end{remark}

\subsection{Density- and compression-sublevel graphs}

Fix one of the vertex functionals used above.  In the raw model this may be
$F(P)=\rho_p^{\Lat}(P)$ or $F(P)=\CRad_D^{\Lat}(P)$.  For a finite smoothing
scheme, put instead
\[
 F_{\Smooth}(P)=\min_{\gamma\in\Smooth(P)}F(\widehat\gamma).
\]
In the notation below, $F(P)$ denotes whichever of these raw or smoothed vertex
costs has been fixed.  For a threshold $\Lambda$, define the induced sublevel graph
\[
 G_{N,\Lambda}^{F}(K;S)
 =G_N^{\Lat,\Move}(K;S)
 \bigl[\{P:F(P)\leq\Lambda\}\bigr].
\]
Its connected components will be called the $F$-admissible components at
parameters $(N,\Lambda)$.  Thus lattice length controls which representatives
are available, while $\Lambda$ controls which of those representatives satisfy
the chosen geometric bound.

\begin{proposition}[Finite two-parameter filtration]
Assume the seed-generated graphs are nested in $N$.  If $N\leq N'$ and
$\Lambda\leq\Lambda'$, then
\[
 G_{N,\Lambda}^{F}(K;S)\subseteq
 G_{N',\Lambda'}^{F}(K;S).
\]
For every finite set of length levels and thresholds, all component births,
component mergers, and first threshold-crossing levels are computable by finite
graph search.
\end{proposition}

\begin{proof}
The length inequality gives inclusion of the underlying seed-generated vertex
sets, and the threshold inequality preserves every vertex satisfying
$F\leq\Lambda$.  The induced edges are inherited from the larger move graph.
All graphs in a fixed finite parameter range are finite, so their connected
components and inclusion-induced mergers can be computed directly.
\end{proof}

\begin{remark}
This two-parameter construction is the discrete counterpart of the density- and
compression-sublevel filtrations proposed in~\cite{OzawaCompression}.  The
length-only merge scales of~\cite{OzawaDiscrete} and the global minimum profiles
of the present paper are one-parameter shadows of this finer object.  For
example, the complete $4_1$ and $6_3$ minimal-layer merge trees determine when
all birth-level components can merge as the length cap grows, while the verified
mirror certificates select extremal paths on which the sublevel
graphs additionally record the density or compression cost required along a
connecting route.
\end{remark}

%==================================================
\section{Summary of unconditional results}\label{sec:unconditional}
%==================================================

We collect the unconditional statements which do not depend on any unproved
continuous-limit assertion or on completion of the planned smoothed numerical
experiments.

\begin{theorem}[Computable smoothed lattice profiles]
Fix a knot type $K$, a periodic lattice $\Lat$, a local move system $\Move$, a
finite seed set $S$, a length cap $N$, and a finite smoothing-perturbation scheme
$\Smooth$.  Then the smoothed seed-generated profiles
\[
 \rho_{p,N}^{\Lat,\Move,\Smooth,\seed}(K;S),
 \qquad
 \CRad_{D,N}^{\Lat,\Move,\Smooth,\seed}(K;S)
\]
are well-defined finite computable quantities.  As $N$ increases, these profiles
are nonincreasing along nested seed-generated searches.
\end{theorem}

\begin{proof}
Well-definedness and finite computability follow from
Theorem~\ref{thm:smoothed-finite-computability}.  Monotonicity is the same
feasible-set inclusion argument as Proposition~\ref{prop:monotonicity}.
\end{proof}

\begin{question}[Stabilized lattice approximation with scaled windows]
Let $\Lat_h=h\ZZ^3$ and let $\Smooth_h$ be a family of smoothing schemes whose
rounding radii, perturbation sizes, and certification tolerances tend to zero with
$h$ while preserving embeddedness and isotopy type.  For a fixed knot type $K$ and
window parameter $\lambda\ge 1$, take the lattice length cap
\[
 N_h(\lambda)=\left\lfloor \frac{\lambda\Rop(K)}{h}\right\rfloor
\]
with the parity adjustment required by the move system.  Under what additional
thickness, approximation, and optimization hypotheses do the smoothed lattice
profiles with cap $N_h(\lambda)$ converge to the corresponding ropelength-windowed
continuous profiles?
\end{question}

%==================================================
\section{Outlook: implementation notes and reproducibility}
%==================================================

The material in this section is an implementation plan and reproducibility
specification, not an additional numerical result of the present paper.  The
purpose is to make clear how the raw pilot computations and the planned smoothed
computations should be archived and labelled.

The following pseudocode describes the intended computation.

\begin{verbatim}
for K in target_knots:
    seeds = load_seed_set(K)
    N0 = initial_length(K)
    for N in admissible_levels(N0, Nmax):
        G = explore_BFACF(seeds, max_length=N)
        best = initialize_records()
        for P in G.vertices:
            for gamma in smooth_perturbations(P, eps, radius, samples):
                thi = thickness(gamma)
                gamma_hat = scale(gamma, 1/thi)
                values = evaluate_density_and_compression(gamma_hat)
                update(best, values, P, gamma)
        write_summary(K, N, best)
\end{verbatim}

For reproducibility, each output record should include:
\begin{enumerate}[label=(\roman*)]
\item knot type and seed file;
\item length cap $N$;
\item number of vertices explored;
\item number of smoothing perturbations per vertex;
\item random seed, if stochastic perturbations are used;
\item best values and the realizing vertex identifiers;
\item whether the search is exhaustive, seed-generated, or capped.
\end{enumerate}

\subsection*{Data and code availability}
The supplementary code and data for the raw-lattice trefoil pilot computations
in Section~\ref{sec:raw-pilots} are archived on Zenodo~\cite{OzawaCompression}:
\begin{center}
\url{https://doi.org/10.5281/zenodo.21319883}.
\end{center}
That archive contains the $24$-edge trefoil seed coordinates, the positive $+2$
plaquette-detour generator, scripts for evaluating $\rho_2$, $\rho_\infty$,
$\CRad_2$, and $\CRad_\infty$, and the summary tables for
$N=24,26,28,30,32$.

The verified BFACF mirror-merge certificates, seed provenance, search summaries,
and independent verification files for the supplied $4_1$ and $6_3$ seed pairs
are archived with the companion move-graph data on
Zenodo~\cite{OzawaDiscreteData}:
\begin{center}
\url{https://doi.org/10.5281/zenodo.21304711}.
\end{center}
The source package accompanying this revision includes the two derived raw
path-profile figures, the complete component summary
\path{tree_geometric_profile_summary.csv}, the $19618$-state table
\path{figure8_tree_state_metrics.csv}, the two-branch $6_3$ table
\path{six3_tree_state_metrics_up_to_N42.csv}, and the C++17 enumerator/evaluator
\path{tree_profiles.cpp}.  The $4_1$ level-$32$ component and both $6_3$
level-$42$ components are exhaustive under the stated quotient.  The final
$6_3$ level-$44$ bounded component is not exhaustively tabulated here; only its
minimal support and selected verified merge certificates are used.  All
computations remain explicitly labelled as raw-lattice, exhaustive bounded,
seed-generated, restricted, or non-exhaustive as appropriate.

%==================================================
\section{Further directions}
%==================================================

Several refinements are natural.

\begin{problem}[Certified thickness for smoothed lattice curves]
Develop a certified algorithm for lower-bounding $\Thi(\gamma)$ for the
specific smoothing schemes used in the experiments.
\end{problem}

\begin{problem}[Further exhaustive minimal-layer profiles]
The raw minimal-layer distributions are now complete for $4_1$ and $6_3$ under
the present quotient.  Extend the computation to further small knot types and
to certified smoothed representatives, and compare the resulting componentwise
profiles across knot types and lattices.
\end{problem}

\begin{problem}[Bounded profile at the top $6_3$ merge]
Exhaustively enumerate the full $N=44$ bounded component of $6_3$ and compute
the minimum, maximum, mean, median, and distribution of the raw and smoothed
density/compression functionals on that root node.  Determine which, if any,
new geometric optima first appear at the top merge rather than already inside
the two level-$42$ branches.
\end{problem}

\begin{problem}[Geometric cost of a merge path]
For a vertex functional $F$, define the minimax cost of connecting two seed
components under a length cap $N$ by minimizing the largest $F$-value attained
along a connecting move path.  Compute this cost for the supplied $4_1$ and
$6_3$ mirror seed pairs and compare it with the ordinary length-only merge
barrier.  Determine whether the minimizing paths for density and compression
can be chosen to coincide.
\end{problem}

%==================================================
\appendix
\section{Coordinate seed used for the trefoil pilot computation}
%==================================================

For reproducibility, we record the 24-step simple-cubic trefoil seed used in
Tables~\ref{tab:pilot-density}--\ref{tab:trefoil-filtered-N}.  The vertices are
listed cyclically, with the final vertex joined to the first by a unit edge.  The
coordinate list is taken from the KnotPlot coordinate page~\cite{KnotPlotCoords};
the use of a 24-step trefoil seed is consistent with Diao's proof that the
smallest nontrivial simple-cubic lattice knots have length 24 and are
trefoils~\cite{DiaoMinimal,DiaoSmallest}.
\begin{verbatim}
0:  ( 1,-1, 1)   1:  ( 1, 0, 1)   2:  ( 1, 1, 1)   3:  ( 1, 2, 1)
4:  ( 0, 2, 1)   5:  (-1, 2, 1)   6:  (-1, 1, 1)   7:  (-1, 0, 1)
8:  (-1, 0, 0)   9:  ( 0, 0, 0)  10:  ( 1, 0, 0)  11:  ( 2, 0, 0)
12: ( 2, 0, 1)  13:  ( 2, 0, 2)  14:  ( 1, 0, 2)  15:  ( 0, 0, 2)
16: ( 0, 0, 1)  17:  ( 0, 1, 1)  18:  ( 0, 1, 0)  19:  ( 0, 1,-1)
20: ( 1, 1,-1)  21:  ( 1, 0,-1)  22:  ( 1,-1,-1)  23:  ( 1,-1, 0)
\end{verbatim}
Using the segment formulas in Section~\ref{sec:raw-pilots}, this coordinate list
gives
\[
 \sD_2^2=\frac{643}{144},\qquad
 \sD_2\approx 2.1131,
\]
which verifies the first trefoil row of Table~\ref{tab:pilot-density}.  The
length-26 candidate in that table is obtained by replacing the edge from vertex
4 to vertex 5 by the detour
\[
(0,2,1)\to (0,3,1)\to(-1,3,1)\to(-1,2,1).
\]

\section*{Acknowledgments}

\paragraph{Use of generative AI.}
The author used ChatGPT (OpenAI) as an interactive aid during the preparation of
this manuscript, including for mathematical discussion, the exploration of
possible formulations and proof strategies, and the improvement of exposition.
All definitions, statements, proofs, computations, code, and references were
independently checked and verified by the author, who takes full responsibility
for the accuracy, originality, and content of the manuscript.

%==================================================
%\section*{Acknowledgements}
%==================================================
%The author thanks the developers of lattice-knot enumeration and BFACF sampling methods, whose work provides the computational background for the present framework.  This manuscript is intended as a self-contained framework paper; larger BFACF experiments and smoothed certified computations will be reported separately.

%==================================================


\begin{thebibliography}{99}

\bibitem{ACPR}
T. Ashton, J. Cantarella, M. Piatek, and E. Rawdon,
\emph{Knot tightening by constrained gradient descent},
Experiment. Math. \textbf{20} (2011), no. 1, 57--90.

\bibitem{BergForster}
B. Berg and D. F{\"o}rster,
\emph{Random paths and random surfaces on a digital computer},
Phys. Lett. B \textbf{106} (1981), 323--326.

\bibitem{ACCF}
C. Arag{\~a}o de Carvalho, S. Caracciolo, and J. Fr{\"o}hlich,
\emph{Polymers and $g|\varphi|^4$ theory in four dimensions},
Nucl. Phys. B \textbf{215} (1983), 209--248.

\bibitem{JvRW}
E. J. Janse van Rensburg and S. G. Whittington,
\emph{The BFACF algorithm and knotted polygons},
J. Phys. A \textbf{24} (1991), 5553--5567.

\bibitem{CKS}
J. Cantarella, R. B. Kusner, and J. M. Sullivan,
\emph{On the minimum ropelength of knots and links},
Invent. Math. \textbf{150} (2002), 257--286.

\bibitem{DiaoMinimal}
Y. Diao,
\emph{Minimal knotted polygons on the cubic lattice},
J. Knot Theory Ramifications \textbf{2} (1993), no. 4, 413--425.

\bibitem{DiaoSmallest}
Y. Diao,
\emph{The number of smallest knots on the cubic lattice},
J. Statist. Phys. \textbf{74} (1994), 1247--1254.

\bibitem{EHL}
P. Exner, E. M. Harrell, and M. Loss,
\emph{Inequalities for means of chords, with application to isoperimetric problems},
Lett. Math. Phys. \textbf{75} (2006), 225--233.

\bibitem{GM}
O. Gonzalez and J. H. Maddocks,
\emph{Global curvature, thickness, and the ideal shapes of knots},
Proc. Natl. Acad. Sci. USA \textbf{96} (1999), 4769--4773.

\bibitem{JansevanRensburg}
E. J. Janse van Rensburg,
\emph{The Statistical Mechanics of Interacting Walks, Polygons, Animals and Vesicles},
Oxford University Press, 2015.

\bibitem{LSDr}
R. A. Litherland, J. Simon, O. Durumeric, and E. Rawdon,
\emph{Thickness of knots},
Topology Appl. \textbf{91} (1999), 233--244.

\bibitem{MinimalKnots}
E. J. Janse van Rensburg and A. Rechnitzer,
\emph{Minimal knotted polygons in cubic lattices},
J. Stat. Mech. Theory Exp. (2011), P09008.

\bibitem{Scharein2009}
R. Scharein, K. Ishihara, J. Arsuaga, Y. Diao, K. Shimokawa, and M. Vazquez,
\emph{Bounds for the minimum step number of knots in the simple cubic lattice},
J. Phys. A: Math. Theor. \textbf{42} (2009), 475006, 24 pp.

\bibitem{Ohara}
J. O'Hara,
\emph{Energy of a knot},
Topology \textbf{30} (1991), 241--247.


\bibitem{KnotPlotCoords}
R. Scharein,
\emph{KnotPlot cubic-lattice knot coordinates},
KnotPlot data page, available at \url{https://knotplot.com/EquiLat/coords.html}
(accessed for the coordinate seed recorded in Appendix A).

\bibitem{OzawaDiscrete}
M.~Ozawa,
\emph{Merge trees of lattice knots},
preprint (2026), arXiv:2605.25322.

\bibitem{OzawaDiscreteData}
M.~Ozawa,
\emph{Supplementary code and data for arXiv:2605.25322},
Zenodo (2026), DOI:
\href{https://doi.org/10.5281/zenodo.21304711}{10.5281/zenodo.21304711}.

\bibitem{OzawaPDensity}
M.~Ozawa,
\emph{Unconstrained and ropelength-windowed $p$-densities of knot types},
preprint (2026), arXiv:2604.23621.

\bibitem{OzawaCompression}
M.~Ozawa,
\emph{Geometric densities and compression radii of knot types},
preprint (2026), arXiv:2604.27912.

\bibitem{OzawaZenodoPDensity}
M.~Ozawa,
\emph{Supplementary code and data for arXiv:2605.29160},
Zenodo, version 3 (2026),
\href{https://doi.org/10.5281/zenodo.21319883}
{doi:10.5281/zenodo.21319883}.

\bibitem{Rawdon}
E. J. Rawdon,
\emph{Approximating the thickness of a knot},
Ideal Knots, Ser. Knots Everything \textbf{19}, World Scientific, 1998, 143--150.

\bibitem{RawdonSimon}
E. J. Rawdon and J. K. Simon,
\emph{Polygonal approximation and energy of smooth knots},
J. Knot Theory Ramifications \textbf{15} (2006), 429--451.

\end{thebibliography}
\end{document}